\documentclass[12pt]{amsart}
\usepackage{amsthm, amssymb, enumerate,color}
\usepackage{tikz}

\usepackage{hyperref}
 \hypersetup{
    pdftitle=   {Classification of real Bott manifolds and acyclic digraphs},
    pdfauthor=  {Suyoung Choi, Mikiya Masuda, and Sang-il Oum}
}

\theoremstyle{plain} \numberwithin{equation}{section}
\newtheorem{theo}{Theorem}[section]
\newtheorem{coro}[theo]{Corollary}
\newtheorem{prop}[theo]{Proposition}
\newtheorem{lemm}[theo]{Lemma}
\theoremstyle{definition}
\newtheorem*{defi}{Definition}
\newtheorem{exam}[theo]{Example}
\theoremstyle{remark}
\newtheorem*{rema}{Remark}
\newtheorem*{note}{Note}

\def\Z{\mathbb Z}

\def\R{\mathbb R}

\def\bB{\bar B}

\def\G{\Gamma}
\def\EA{\mathcal E_A}
\def\EB{\mathcal E_B}
\def\tEA{\bar{\mathcal E}_A}
\def\tEB{\bar{\mathcal E}_B}

\def\Q{\mathbb Q}
\def\ZZ{\Z/2}

\def\f{\tilde f}
\def\MH{\mathcal H}

\def\T{\mathfrak B}
\def\bA{\bar A}

\def\OP{(\text{Op1})}
\def\OC{(\text{Op2})}
\def\OR{(\text{Op3})}

\DeclareMathOperator{\rank}{rank}

\DeclareMathOperator{\row}{row}

\def\Phik{\Phi_k}
\def\Philk{\Phi^{\ell,m}}
\def\phik{\phi_k}
\def\philk{\phi^{\ell,m}}
\def\fk{\tilde f_k}
\def\flk{\tilde f^{\ell,m}}

\def\slides#1#2{\diamond #1 #2}
\def\G{\Gamma}
\def\f{\tilde f}
\def\cormatrix#1#2{[#1]_{#2}}
\def\Diff#1{\mathcal D_#1}
\def\Ori#1{\mathcal O_#1}
\def\Sym#1{\mathcal S_#1}

\def\abs#1{\lvert #1\rvert}

\def\arxiv#1{\href{http://arXiv.org/abs/#1}{arXiv:#1}}
\def\MR#1{\href{http://www.ams.org/mathscinet-getitem?mr=#1}{MR #1}}

\begin{document}

\title[Real Bott manifolds and acyclic digraphs] {Classification of real Bott manifolds and acyclic digraphs}

\author{Suyoung Choi}
\address{Department of Mathematics, Ajou University, San 5, Woncheondong, Yeongtonggu, Suwon 16499, Republic of Korea}
\email{schoi@ajou.ac.kr}

\author{Mikiya Masuda}
\address{Department of Mathematics, Osaka City University, Sumiyoshi-ku, Osaka 558-8585, Japan.}
\email{masuda@sci.osaka-cu.ac.jp}

\author{Sang-il Oum}
\address{Department of Mathematical Sciences, KAIST, 291 Daehakro, Yuseong-gu, Daejeon 34141, Republic of Korea}
\email{sangil@kaist.edu}

\thanks{The first author was supported by Basic Science Research Program through the National Research Foundation of Korea(NRF) funded by the Ministry of Education(NRF-2011-0024975) and TJ Park Science Fellowship.
  The second author was partially supported by Grant-in-Aid for Scientific Research 19204007.
  The third author was supported by Basic Science Research Program through the National Research Foundation of Korea(NRF) funded by the Ministry of Science, ICT \& Future Planning (2011-0011653).}


\date{\today}
\maketitle

\begin{abstract}
  We completely characterize real Bott manifolds up to affine diffeomorphism in terms of three simple matrix operations on square binary matrices obtained from strictly upper triangular matrices by permuting rows and columns simultaneously.
  We also prove that any graded ring isomorphism between the cohomology rings of real Bott manifolds with $\ZZ$ coefficients is induced by an affine diffeomorphism between the real Bott manifolds.

  Our characterization can also be described in terms of graph operations on directed acyclic graphs.
  Using this combinatorial interpretation, we prove that the decomposition of a real Bott manifold into a product of indecomposable real Bott manifolds is unique up to permutations of the indecomposable factors.
  Finally, we produce some numerical invariants of real Bott manifolds from the viewpoint of graph theory and discuss their topological meaning. As a by-product, we prove that the toral rank conjecture holds for real Bott manifolds.
\end{abstract}

\maketitle

\section{Introduction} \label{sect:intro}
A manifold  $M$ is called a \emph{real Bott manifold} if there is a sequence of $\R P^1$ bundles
\begin{equation} \label{tower}
    M=M_n\stackrel{\R P^1}\longrightarrow M_{n-1}\stackrel{\R P^1}\longrightarrow \cdots\stackrel{\R P^1}\longrightarrow M_1 \stackrel{\R P^1}\longrightarrow M_0=\{\textrm{a point}\},
\end{equation}
such that for each $j\in \{1,2,\ldots,n\}$, $M_j\to M_{j-1}$ is the projective bundle of the Whitney sum of a real line bundle $L_{j-1}$ and the trivial real line bundle over $M_{j-1}$.
The sequence \eqref{tower} is called a \emph{real Bott tower} of height $n$, and it is a real analogue of a Bott tower introduced by Grossberg and Karshon~\cite{gr-ka94}.
A real Bott manifold naturally supports an action of an elementary abelian $2$-group.
In fact, Kamishima and Masuda~\cite{ka-ma09} proved that a manifold is a real Bott manifold if and only if it is a real toric manifold admitting a flat riemannian metric invariant under the action.

It is well known that real line bundles are classified by their first Stiefel-Whitney classes.
With the binary field $\ZZ=\{0,1\}$,  $H^1(M_{j-1};\ZZ)$ is isomorphic to $(\ZZ)^{j-1}$ through a canonical basis. Therefore the line bundle $L_{j-1}$ is determined by a vector $A_j$ in $(\ZZ)^{j-1}$.
By regarding $A_j$ as a column vector in $(\ZZ)^n$ by adding zeros, we obtain a strictly upper triangular  $n\times n$ matrix $A$ whose $j$-th column vector is $A_j$.
Since the real Bott manifold $M_n$ is determined by the matrix $A$, we may denote it by $M(A)$.
We shall give an alternative construction of $M(A)$ in Section~\ref{sect:rbott}, which also make sense for \emph{Bott matrices} $A$. A square matrix $A$ is a \emph{Bott matrix} if $A=PBP^{-1}$ for a permutation matrix $P$ and an strictly upper triangular binary matrix $B$. Let $\T(n)$ be the set of all $n\times n$ Bott matrices.


Two different Bott matrices $A$ and $B$ may produce (affinely) diffeomorphic real Bott manifolds.
How can we determine whether $M(A)$ and $M(B)$ are affinely diffeomorhic?
By analyzing graded ring isomorphisms between cohomology rings of real Bott manifolds, we found three operations $\OP$, $\OC$, and $\OR$ on Bott matrices whose compositions induce all graded ring isomorphisms. Our first main theorem shows that these operations actually characterize affinely diffeomorphic real Bott manifolds, answering the above question. We say that two Bott matrices  are \emph{Bott equivalent} if one is transformed to the other through a sequence of these three operations.
%

\begin{theo} \label{main}
    The following are equivalent for Bott matrices $A,B$ in $\T(n)$.
    \begin{enumerate}[(1)]
    \item $A$ and $B$ are Bott equivalent.
    \item $M(A)$ and $M(B)$ are affinely diffeomorphic.
    \item $H^*(M(A);\ZZ)$ and $H^*(M(B);\ZZ)$ are isomorphic as graded rings.
    \end{enumerate}
    Moreover, every graded ring isomorphism from $H^*(M(A);\ZZ)$ to $H^*(M(B);\ZZ)$ is induced by an affine diffeomorphism from $M(B)$ to $M(A)$.
\end{theo}
In particular, we obtain the following main theorem of Kamishima and Masuda~\cite{ka-ma09} as a corollary.

\begin{coro}\label{maincoro}
    Two real Bott manifolds are diffeomorphic if and only if their cohomology rings with $\ZZ$ coefficients are isomorphic as graded rings.
\end{coro}

To a Bott matrix $A$, one can associate an \emph{acyclic digraph} (a directed graph with no directed cycles) whose adjacency matrix is $A$.
This correspondence is a bijection from $\T(n)$ to the set of acyclic digraphs on vertices $\{1,2,\ldots,n\}$.
Through the bijection, the three operations $\OP$, $\OC$ and $\OR$ on $\T(n)$ can be described as operations on acyclic digraphs.
$\OP$ corresponds to permuting labels of vertices.
To our surprise, $\OC$ corresponds to a known operation in graph theory called a \emph{local complementation} while the operation corresponding to $\OR$ seems not studied and we call it a \emph{slide}.
As far as we know, a local complementation on digraphs was first introduced by Bouchet~\cite{Bouchet1987c}.
Fon-Der-Flaass~\cite{fon94} surveyed this operation.
This operation also appears in the coding theory \cite{da-pa09} and quantum information theory \cite{va-de-de04}.
Our result adds another application of this operation in topology.

We prove that real Bott manifolds of dimension $n$ up to diffeomorphism can be identified with non-isomorphic acyclic digraphs on $n$ vertices up to local complementation and slide.
This combinatorial interpretation enables us to efficiently count the number $\Diff{n}$ of real Bott manifolds of dimension $n$ up to diffeomorphism.
We list $\Diff{n}$ in Table~\ref{tab:Bott equivalence} for $n \leq 8$.
Previously, $\Diff{n}$ was known for $n\leq 5$ and it was a hard task to find $\Diff{5}$ using a geometrical method (\cite{nazr10}).
The computation of $\Diff{8}$ takes less than $10$ minutes by a regular desktop computer if we use the list of non-isomorphic acyclic digraphs provided by B. D. McKay.\footnote{
\url{http://cs.anu.edu.au/~bdm/data/digraphs.html}}
In addition to $\Diff{n}$, we also list the number $\Ori{n}$ and  $\Sym{n}$ of $n$-dimensional orientable and symplectic, respectively, real Bott manifolds in Table~\ref{tab:Bott equivalence} for small values of $n$.
\begin{table}
    \centering
    \begin{tabular}{c|c c c c c c c c c c }
        \hline
        $n$ & 1 & 2 & 3 & 4  & 5 & 6 & 7 & 8&9&10\\ \hline
        $\Diff{n} $ & 1 & 2 & 4 & 12 & 54 & 472 & 8,512 & 328,416&?&?\\ \hline
        $\Ori{n}$ & 1 & 1 & 2 & 3 & 8 & 29 & 222 & 3,607&131,373&?\\ \hline
        $\Sym{n} $ & 0  & 1 & 0  & 2 &  0 & 6  &  0   &  31 &0&416\\
        \hline
    \end{tabular}
    \caption{The numbers $\Diff{n},~\Ori{n},~\Sym{n}$ of $n$-dimensional real Bott manifolds, orientable real Bott manifolds and symplectic real Bott manifolds up to diffeomorphism, respectively.}
    \label{tab:Bott equivalence}
\end{table}

Our classification of real Bott manifolds helps us to prove the topologically unique decomposition property for real Bott manifolds as follows.
We say that a real Bott manifold is \emph{indecomposable} if it is not diffeomorphic to a product of more than one real Bott manifolds.

\begin{theo}[Unique Decomposition Property] \label{main1}
    The decomposition of a real Bott manifold into a product of indecomposable real Bott manifolds is unique up to permutations of the indecomposable factors.
\end{theo}
In particular, since $S^1$ is a real Bott manifold $\R P^1$, we have the following corollary.
\begin{coro}[Cancellation Property]  \label{main1coro}
    Let $M$ and $M'$ be real Bott manifolds. If $S^1\times M$ and $S^1\times M'$ are diffeomorphic, then $M$ and $M'$ are diffeomorphic.
\end{coro}

As remarked before, a real Bott manifold admits a flat riemannian metric.
We note that the cancellation property above fails to hold for general compact flat riemannian manifolds \cite{char65-1}.
It would be interesting to ask whether Theorem~\ref{main1} and Corollary~\ref{main1coro} hold for any (real) toric manifolds.

Our combinatorial interpretation allows us to discover several numerical invariants of real Bott manifolds up to diffeomorphism.
Interestingly, those invariants can be thought of as a refinement of the topological and geometrical properties of real Bott manifolds.
We will discuss them in Section~\ref{sec:invariant}.
In particular, we prove the following theorem, which confirms that the toral rank conjecture \cite[p.280]{al-pu93} holds for real Bott manifolds.
\begin{theo} \label{thm:trc}
Let $A\in\T(n)$.
If $M(A)$ admits an effective topological action of a torus $T^k$ of dimension $k$, then
$$\sum_{i=0}^n \dim_{\Q} H^i(M(A);\Q)\ge 2^k.$$
\end{theo}

This paper is organized as follows.
In Section~\ref{sect:rbott} we describe $M(A)$ and its cohomology ring
explicitly in terms of a Bott matrix $A$.
We introduce three operations on $\T(n)$ in Section~\ref{sect:matrix} and translate them into operations on acyclic digraphs in Section~\ref{section:acyclic digraphs}.
To each operation we associate an affine diffeomorphism between real Bott manifolds in Section~\ref{sect:affine}, which implies the implication (1) $\Rightarrow$ (2) in Theorem~\ref{main}.
The implication (2) $\Rightarrow$ (3) is trivial.
In Section~\ref{sect:cohom} we prove the last part of Theorem~\ref{main}.
The argument also establishes the implication (3) $\Rightarrow$ (1).
We prove Theorem~\ref{main1} in Section~\ref{sect:decom}.
In Section~\ref{sec:invariant} we produce numerical invariants of real Bott manifolds from the viewpoint of graph theory. In particular, in Section~\ref{subsection:rank}, we prove Theorem~\ref{thm:trc}.

\begin{note}
    This paper is a combination of preprints \cite{Ch-Ou-2010} and \cite{masu08}.
    After the second author wrote the paper \cite{masu08}, the first and third authors wrote the paper \cite{Ch-Ou-2010} which relates results in \cite{masu08} to acyclic digraphs based on the observation in \cite{Choi-2008}, simplifies the operation $\OR$ in \cite{masu08} and produces many numerical invariants of real Bott manifolds.
    In this paper, we also re-prove Theorem~\ref{main1} from a graph theoretical viewpoint.
    Readers may find an algebraic proof of Theorem~\ref{main1} in \cite{masu08}.
    These two proofs are completely different.
    Some parts of this paper including Theorem~\ref{thm:trc} are neither in \cite{masu08} nor in \cite{Ch-Ou-2010}.
    We hope that this combination will make the subject and results more appealing to the readers.
\end{note}

\section{Real Bott manifolds and their cohomology rings} \label{sect:rbott}
The real Bott manifold $M(A)$ associated with a strictly upper triangular $n\times n$ binary matrix $A$ can be described as the quotient of the $n$-dimensional torus by a free action of an elementary abelian $2$-group of rank $n$.
The free action is uniquely determined by the matrix $A$.
In addition, this quotient construction also works if $A$ is conjugate by a permutation matrix to a strictly upper triangular binary matrix.
Motivated by this, we make the following definition.

\begin{defi}
    A square matrix $A$ is a \emph{Bott matrix} if
    $$
        A=PBP^{-1}
    $$
    for a permutation matrix $P$ and a strictly upper triangular binary matrix $B$.
    In other words, a Bott matrix is conjugate by a permutation matrix to a strictly upper triangular binary matrix.
    We denote by $\T(n)$ the set of all $n\times n$ Bott matrices.\footnote{In \cite{masu08} $\T(n)$ is defined to be the set of strictly upper triangular $n\times n$ binary matrices.}
\end{defi}

Masuda and Panov~\cite[Lemma 3.3]{ma-pa08} showed that a binary square matrix  $A$ is a Bott matrix if and only if every principal minor of $A+I$ is $1$ over $\ZZ=\{0,1\}$, where $I$ is the identity matrix.

Let us recall the quotient construction and the structure of the cohomology ring of $M(A)$ for $A\in \T(n)$.
Let $S^1$ denote the unit circle consisting of complex numbers with absolute value $1$.
For $z\in S^1$ and $a\in \ZZ$, we use the following notation
$$
    z(a):=\begin{cases} z \quad&\text{if $a=0$,}\\
    \bar z\quad&\text{if $a=1$}.
    \end{cases}
$$
For a matrix $A$, let  $A^i_j$ be the $(i,j)$ entry of $A$ and let $A^i$, $A_j$ be the $i$-th row vector, the $j$-th column vector, respectively, of $A$.
We define the involutions $a_1,a_2,\ldots,a_n$ on $T^n:=(S^1)^n$ by
\begin{equation} \label{ai}
    a_i (z_1,\dots,z_n):=(z_1(A^i_1),\dots,z_{i-1}(A^i_{i-1}),-z_i,z_{i+1}(A^i_{i+1}),\dots, z_n(A^i_n)).
\end{equation}
These involutions $a_1,a_2,\ldots,a_n$ commute with each other and generate an elementary abelian $2$-group of rank $n$, denoted by $G(A)$.

\begin{lemm} \label{free}
    The action of $G(A)$ on $T^n$ is free.
\end{lemm}

\begin{proof}
    Let $A\in\T(n)$.
    If $A$ is strictly upper triangular, then $z_1(A^i_1)=z_1,\dots,z_{i-1}(A^i_{i-1})=z_{i-1}$ in \eqref{ai} because $A^i_1=\dots=A^i_{i-1}=0$.
    Therefore, for an element $t = a_{i_1} \cdots a_{i_\ell}$ of $G(A)$ with $i_1 < \cdots <i_\ell$, the $i_1$-th component of $t (z_1, \ldots, z_n)$ is $-z_{i_1}$.
    Hence, the action of $G(A)$ on $T^n$ is clearly free when $A$ is strictly upper triangular.

    Now let us assume that $A$ is not strictly upper triangular.
    There is a permutation $\sigma$ on $\{1,2,\ldots,n\}$ with its permutation matrix $P$ such that $B=PAP^{-1}$ is strictly upper triangular, where
$$
    P^i_j = \left\{
              \begin{array}{ll}
                1 & \hbox{if $i= \sigma(j)$,} \\
                0 & \hbox{otherwise.}
              \end{array}
            \right.
$$
    Since $PA=BP$, we have
\begin{equation} \label{SA=BA1}
    A^i_j=(PA)^{\sigma(i)}_j=(BP)^{\sigma(i)}_j=B^{\sigma(i)}_{\sigma(j)}.
\end{equation}
    This together with \eqref{ai} means that if we change the suffix of the coordinate $(z_1,\dots,z_n)$ by $\sigma$, then the involution $a_i$ in \eqref{ai} is the same as the involution $b_{\sigma(i)}$ associated with $B$ for each $i$.
    Since the action of $G(B)$ on $T^n$ is free as $B$ is strictly upper triangular, so is the action of $G(A)$ on $T^n$.
\end{proof}

We define $M(A)$ to be the orbit space $T^n/G(A)$.
By Lemma~\ref{free}, $M(A)$ is a closed smooth manifold.

\begin{rema}
We remark that $M(A)$ is a flat riemannian manifold.
In fact, the Euclidean motions $s_1,s_2,\ldots,s_n$ on $\R^n$ defined by
$$
    s_i(u_1,\dots,u_n):=((-1)^{A^i_1}u_1,\dots,(-1)^{A^i_{i-1}}u_{i-1}, u_i+\frac{1}{2}, (-1)^{A_{i+1}^i}u_{i+1},\dots, (-1)^{A_{n}^i}u_n)
$$
generate a crystallographic group $\G(A)$, where the subgroup generated by $s_1^2,\dots,s_n^2$ consists of all translations by $\Z^n$, and $\G(A)/\Z^n=G(A)$.
The action of $\G(A)$ on $\R^n$ is free.
Through the identification $\R/\Z$ with $S^1$ via an exponential map \[u\mapsto \exp(2\pi\sqrt{-1}u),\] the orbit space $\R^n/\Z^n$ agrees with $T^n$ and the orbit space $\R^n/\G(A)$ agrees with $M(A)=T^n/G(A)$.
\end{rema}

For $k=1,2,\ldots,n$, let $G_k$ be the subgroup of $G(A)$ generated by $a_1,\dots,a_k$.
Obviously $G_n=G(A)$.
Let $T^k:=(S^1)^k$ be the product of first $k$-factors in $T^n=(S^1)^n$.
Then $G_k$ acts on $T^k$ by restricting the action of $G_k$ on $T^n$ to $T^k$ and the orbit space $T^k/G_k$ is a manifold of dimension $k$.
If $A$ is strictly upper triangular, then the natural projections $T^k\to T^{k-1}$ for $k=1,2,\ldots,n$ produce a real Bott tower
$$
    M(A)=T^n/G_n\to T^{n-1}/G_{n-1} \to \dots\to T^1/G_1\to \text{\{a point\}},
$$
which agrees with \eqref{tower} in Section~\ref{sect:intro} (see \cite{ka-ma09}).

The graded ring structure of $H^*(M(A);\ZZ)$ can be described explicitly in terms of the matrix $A$.
We shall recall it.
For a homomorphism $\lambda\colon G(A)\to \{\pm 1\}$ we denote by $\R(\lambda)$ the real one-dimensional $G(A)$-module associated with $\lambda$.
Then the orbit space of $T^n\times \R(\lambda)$ by the diagonal action of $G(A)$, denoted by $L(\lambda)$, defines a real line bundle over $M(A)$ with the first projection.
For $j\in \{1,2,\ldots,n\}$, let $\lambda_j\colon G(A)\to \{\pm 1\}$ be a homomorphism such that
$$
    \lambda_j(a_i)=
    \begin{cases}
      -1 &\text{if }i=j,\\
      1 & \text{otherwise.}
    \end{cases}
$$
We set
$$
    x_j=w_1(L(\lambda_j))
$$
where $w_1$ denotes the first Stiefel-Whitney class.

\begin{lemm} \label{cohoA}
    Let  $A$ be a Bott matrix in $\T(n)$.
    Then
$$
    H^*(M(A);\ZZ)=\ZZ[x_1,\dots,x_n]/(x_j^2=x_j\sum_{i=1}^nA^i_jx_i\mid j=1,\ldots,n)
$$
    as graded rings.
    Moreover,
    \begin{enumerate}[(i)]
        \item $M(A)$ is orientable if and only if $\sum_{j=1}^nA^i_j=0$ in $\ZZ$ for every $i\in \{1,2,\ldots,n\}$,
        \item $M(A)$ admits a symplectic form if and only if $\lvert\{ k\mid A_k=A_j\}\rvert$ is even for every $j\in\{1,2,\ldots,n\}$.
    \end{enumerate}
\end{lemm}

\begin{proof}
  Since $A\in \T(n)$ is conjugate by a permutation matrix to a strictly upper triangular matrix, we may assume that $A$ is strictly upper triangular (see the proof of Lemma~\ref{free}).
  Then the first two statements are proved in \cite[Lemmas 2.1 and 2.2]{ka-ma09} and the last statement is proved in \cite{Ishida-10}.
\end{proof}

Let $A$, $B$ be Bott matrices in $\T(n)$.
Since $M(A)=T^n/G(A)$ and $M(B)=T^n/G(B)$, an affine automorphism $\f$ of $T^n$ together with a group isomorphism $\phi\colon G(B)\to G(A)$ induces an affine diffeomorphism $f\colon M(B)\to M(A)$ if $\f$ is \emph{$\phi$-equivariant}, that is
$$
    \f(gz)=\phi(g)\f(z) \text{ for }g\in G(B)\text{ and }z\in T^n.
$$
Since the actions of $G(A)$ and $G(B)$ on $T^n$ are free, the isomorphism $\phi$ will be uniquely determined by $\f$ if it exists.
We shall use $b_i$ and $y_j$ for $M(B)$ in place of $a_i$ and $x_j$ for $M(A)$.
\begin{lemm} \label{f*}
For Bott matrices $A,~B$ in $\T(n)$, let $f \colon M(B) \to M(A)$ be the affine diffeomorphism induced by a $\phi$-equivariant affine automorphism $\f$ of $T^n$, where $\phi\colon G(B) \to G(A)$ is a group isomorphism.
If $\phi(b_i)=\prod_{j=1}^na_j^{F^i_j}$ with $F^i_j\in \ZZ$,
then $f^*(x_j)=\sum_{i=1}^nF^i_jy_i$.
\end{lemm}

\begin{proof}
A map $T^n\times \R(\lambda\circ\phi)\to T^n\times \R(\lambda)$ sending $(z,u)$ to $(\f(z),u)$ induces a bundle map $L(\lambda\circ\phi) \to L(\lambda)$ covering $f\colon M(B)\to M(A)$.
Since $(\lambda_j\circ\phi)(b_i)=F^i_j$, this implies the lemma.
\end{proof}

\section{Three matrix operations} \label{sect:matrix}

In this section we introduce three operations on Bott matrices used in later sections to analyze when $M(A)$ and $M(B)$ are diffeomorphic and when $H^*(M(A);\ZZ)$ and $H^*(M(B);\ZZ)$ are isomorphic for two Bott matrices $A$, $B$ in $\T(n)$.

\subsection*{Operation $\OP$.}
For a permutation matrix $P$ of a permutation $\sigma$ on $\{1,2,\ldots,n\}$, we define a map $\Phi_P$ on $n\times n$ matrices such that
$$
    \Phi_P(A):=PAP^{-1}.
$$
Thus if we set $B=\Phi_P(A)$, then
\begin{equation} \label{SA=BA}
    A^i_j=B^{\sigma(i)}_{\sigma(j)}
\end{equation}
as observed in \eqref{SA=BA1}.

\subsection*{Operation $\OC$.}
For $k\in\{1,2,\ldots,n\}$, we define $\Phik$ to be the operation which adds the $k$-th column of an $n\times n$ matrix to every column having $1$ in the $k$-th row.
In other words, for an $n\times n$ binary matrix $A$, the $n\times n$ matrix $\Phik(A)$ is given by
\begin{equation} \label{2nd}
    \Phik(A)_j:=A_j+A^k_jA_k\quad\text{for $j\in \{1,2,\ldots,n\}$}.
\end{equation}
We note that if $A\in \T(n)$, then $\Phik(A)\in \T(n)$.
In fact, if $A$ is strictly upper triangular, then so is $\Phi_k(A)$, and the general case reduces to the strictly upper triangular case by \eqref{SA=BA}.
Since every diagonal entry of a Bott matrix $A$ is zero, $(\Phi_k\circ \Phik)(A)=A$ and therefore $\Phi_k$ is a bijection on $\T(n)$.

\subsection*{Operation $\OR$.}
For distinct $\ell, m$ in $\{1,2,\ldots,n\}$, we define $\Philk$ on $n\times n$ matrices $A$ with $A_\ell=A_m$ to be the operation which adds the $\ell$-th row to the $m$-th row.
In other words, $\Philk(A)$ is defined to be an $n\times n$ matrix by
\begin{equation} \label{3rd}
    \Philk(A)^i:=\begin{cases} A^m + A^\ell &\text{if  $i = m$ and $A_\ell=A_m$,}\\
    A^i&\text{otherwise.}
    \end{cases}
\end{equation}
Since the diagonal entries of a Bott matrix $A$ are all zero, the condition $A_\ell=A_m$ implies that
\begin{equation}\label{eqn:masuda-added}
    0 = A^\ell_\ell=A^\ell_m \quad \text{ and } \quad A^m_\ell = A^m_m = 0.
\end{equation}
and one can check that $\Philk(A)$ stays in $\T(n)$. In fact, if $A$ is upper triangular, $\Phi^{\ell,m}(A)$ is upper triangular when $\ell > m$ and is conjugate to an upper triangular matrix by the transposition $(\ell, m)$ when $\ell < m$. The general case reduces to the strictly upper triangular case by \eqref{SA=BA}.

Note that if $A_\ell=A_m$, then $\Philk(A)_\ell = \Philk(A)_m$ and $(\Philk \circ \Philk)(A)=A$.

\begin{defi}
    Two Bott matrices in $\T(n)$ are \emph{Bott equivalent} if one can be transformed into the other through a sequence of the three operations $\OP$, $\OC$ and $\OR$.
\end{defi}
We note that every Bott equivalence class in $\T(n)$ has a representative of a strictly upper triangular matrix (not necessarily unique) because of the operation $\OP$.

\begin{exam} \label{matrix34}
  There are two $2\times 2$ strictly upper triangular binary matrices and they are not Bott equivalent.
  There are $2^3=8$ strictly upper triangular binary matrices of size $3$ and they are classified into the following four Bott equivalence classes:
\begin{enumerate}[(1)]
    \item The zero matrix of size $3$
    \item
    ${\tiny
    \begin{pmatrix}
        0 & 1 & 0\\
        0 & 0 & 0\\
        0 & 0 & 0
    \end{pmatrix}\quad
    \begin{pmatrix}
        0 & 0 & 1\\
        0 & 0 & 0\\
        0 & 0 & 0
    \end{pmatrix}\quad
    \begin{pmatrix}
        0 & 0 & 0\\
        0 & 0 & 1\\
        0 & 0 & 0
    \end{pmatrix}\quad
    \begin{pmatrix}
        0 & 0 & 1\\
        0 & 0 & 1\\
        0 & 0 & 0
    \end{pmatrix}
    }$
    \item
    ${\tiny
    \begin{pmatrix}
        0 & 1 & 1\\
        0 & 0 & 0\\
        0 & 0 & 0
    \end{pmatrix}
    }$
    \item
    ${\tiny
    \begin{pmatrix}
        0 & 1 & 0\\
        0 & 0 & 1\\
        0 & 0 & 0
    \end{pmatrix}\quad
    \begin{pmatrix}
        0 & 1 & 1\\
        0 & 0 & 1\\
        0 & 0 & 0
    \end{pmatrix}
    }$
\end{enumerate}
  There are $2^6=64$ strictly upper triangular binary matrices of size $4$ and they are classified into $12$ Bott equivalence classes, see \cite{ka-ma09} and \cite{nazr08}.
  Furthermore, there are $2^{10}=1024$ strictly upper triangular binary matrices of size $5$ and they are classified into $54$ Bott equivalence classes; see Table~\ref{tab:Bott equivalence} in Section~\ref{sect:intro}.
\end{exam}

\begin{exam} \label{Deltan}
  Let $\Delta(n)$ be the set of all $n\times n$ strictly upper triangular binary matrices $A$ such that $A^1_2=A^2_3=\cdots=A^{n-1}_n=1$.
  One can change the $(i,i+2)$ entry into $0$ for $i=1,\ldots,n-2$ using the operation $\OC$, so that $A$ is Bott equivalent to the matrix $\bA$ of the following form
$$
\footnotesize
\bA=\begin{pmatrix}
0&1&0&\bA^1_4&\bA^1_5&\dots&\bA^1_{n-1}&\bA^1_n\\
0&0&1&0&\bA^2_5&\dots&\bA^2_{n-1}&\bA^2_n\\
\vdots&\vdots& \ddots &\ddots &\ddots&\ddots &\vdots & \vdots\\
0& 0& \dots&0&1 &0 & \bA^{n-4}_{n-1} & \bA^{n-4}_n\\
0& 0& \dots&0&0  &1 & 0 & \bA^{n-3}_n\\
0& 0& \dots&0&0  &0 & 1 & 0\\
0& 0& \dots&0&0  &0 & 0 & 1\\
0& 0& \dots&0&0  &0 & 0 & 0
\end{pmatrix}.
$$
The matrix $\bA$ is uniquely determined by $A$ and two matrices $A,B\in \Delta(n)$ are Bott equivalent if and only if $\bA=\bB$.
Therefore $\Delta(n)$ has exactly $2^{(n-2)(n-3)/2}$ Bott equivalence classes for $n\ge 2$.
\end{exam}

\section{Acyclic digraphs} \label{section:acyclic digraphs}
In this section, we spell out what the operations $\OC$ and $\OR$ correspond to under the translation between binary matrices and directed graphs.
A \emph{directed graph} (simply \emph{digraph}) $D$ consists of a finite set $V(D)$ of elements called \emph{vertices} and a set $A(D)$ of ordered pairs of distinct vertices called \emph{arcs}.
Two digraphs $D$ and $H$ are \emph{isomorphic} if there is a bijection $\psi\colon V(D)\to V(H)$ such that $(u,v)\in A(D)$ if and only if $(\psi(u),\psi(v))\in A(H)$.
If $(u,v)\in A(D)$, then $v$ is called an \emph{out-neighbor} of $u$ and $u$ is called an \emph{in-neighbor} of $v$.
For a vertex $v$ of $D$, we denote by $N_D^+(v)$ and $N_D^-(v)$ the set of all out-neighbors and in-neighbors, respectively, of $v$.
The \emph{out-degree}  $\deg_D^+(v)$ and the \emph{in-degree} $\deg_D^-(v)$ of $v\in V(D)$ are the number of out-neighbors and in-neighbors, respectively, of $v$.
An ordering $v_1,v_2,\ldots,v_n$ of the vertices is called \emph{acyclic} if $i<j$ whenever $(v_i,v_j) \in A(D)$.
A digraph is called \emph{acyclic} if it admits an acyclic ordering.
Equivalently, a digraph is acyclic if and only if it has no directed cycle; see
\cite[Proposition 1.4.3]{baje-gu01}.

For a digraph $D$ with an ordering of the vertices $v_1,v_2,\ldots,v_n$, the \emph{adjacency matrix} $A_D$ of $D$ is an $n\times n$ binary matrix such that $(A_D)^i_j=1$ if and only if $(v_i,v_j)\in A(D)$.
To a Bott matrix $A$ in $\T(n)$, we associate a digraph $D_A$ on $n$ vertices $\{v_1,v_2,\ldots,v_n\}$ in such a way that $(v_i,v_j)$ is an arc of $D_A$ if and only if $A^i_j=1$.
Equivalently $D_A$ is the digraph whose adjacency matrix is $A$.
In other words, $v_j$ is an out-neighbor of $v_i$ in $D_A$ if and only if $A^i_j=1$.
Therefore,
\[
\begin{split}
N_{D_A}^+(v_i)=\{ v_j\mid A^i_j=1\} \quad &\text{and}\quad
\deg_{D_A}^+(v_i)=\lvert\{ j \mid A^i_j=1\}\rvert, \\
N_{D_A}^-(v_j)=\{ v_i\mid A^i_j=1\} \quad &\text{and}\quad \deg_{D_A}^-(v_j)=\lvert\{ i \mid A^i_j=1\}\rvert,
\end{split}
\]
and the statements (i) and (ii) in Lemma~\ref{cohoA} can be translated as follows.

\begin{lemm} \label{ori-sym}
Let $A\in \T(n)$.  Then
\begin{enumerate}[(i)]
\item $M(A)$ is orientable if and only if
every vertex of $D_A$ has an even out-degree.
\item $M(A)$ admits a symplectic form if and only if
for each vertex  $v$ of $D_A$,
there are even number of vertices (including $v$ itself) having the
same set of in-neighbors with $v$.
\end{enumerate}
\end{lemm}
We claim that $D_A$ is acyclic for each Bott matrix $A$ in $\T(n)$.
If $A$ is strictly upper triangular, this is obvious because $v_1,v_2,\ldots,v_n$ is an acyclic ordering.
When $A$ is conjugate to a strictly upper triangular matrix $B$ by a permutation matrix of a permutation $\sigma$ on $\{1,2,\ldots,n\}$, then $A^i_j=B^{\sigma(i)}_{\sigma(j)}$ as observed in \eqref{SA=BA1} and therefore $D_A$ is isomorphic to $D_B$ which is acyclic.

The mapping from $\T(n)$ to the set of acyclic digraphs on fixed $n$ vertices $\{v_1,v_2,\ldots,v_n\}$ is bijective.
Therefore, the three operations introduced in Section~\ref{sect:matrix} can be translated and visualized as operations on acyclic digraphs.
It is easy to see that the operation $\OP$ corresponds to the isomorphism of graphs.

In the following, we will discuss operations corresponding to $\OC$ and $\OR$.
For sets $X$ and $Y$, we denote $(X\setminus Y)\cup (Y\setminus X)$ by $X\Delta Y$.

\subsection*{Local complementation.}
Let $D$ be a digraph.
For $v\in V(D)$, we define $D*v$ to be the digraph with $V(D*v)=V(D)$ and
$$
    A(D*v) = A(D) \Delta \{ (u,w) \in N_D^-(v) \times N_D^+(v)\}.
$$
Namely, $(u,w)\in N_D^-(v)\times N_D^+(v)$ is removed from $D$ if it is an arc of $D$ and added to $D$ otherwise.
The operation to obtain $D*v$ from $D$ is called the \emph{local complementation} at $v$. See Figure~\ref{fig:lc}.

\begin{figure}
  \tikzstyle{v}=[circle, draw, solid, fill=black!50, inner sep=0pt, minimum width=4pt]
  \tikzstyle{every edge}=[->,>=stealth,draw]
 \begin{tikzpicture}[thick,scale=0.5]
    \node [v] (bot1) {};
    \node [v] (bot2) [right of=bot1]{};
    \node [v] (bot3) [right of=bot2]{};
    \node[v] (mid1) [above of=bot1]{}
    edge[->] (bot2);
    \node[v] (mid2) [right of=mid1][label=right:$v$] {}
    edge[->] (bot1)
    edge[->] (bot2);
   \node [v](mid3) [right of=mid2] {}
    edge[->] (bot3);
    \node [v] (top1) [above right of=mid1] {}
    edge[->] (bot1)
    edge[->] (mid1)
    edge[->] (mid2);
    \node [v] (top2) [right of=top1] {}
    edge[->] (mid1)
    edge[->] (mid2)
    edge[->] (mid3)
    edge[->] (bot2);
    \node [below of=bot2,node distance=3mm] {$D$};
\end{tikzpicture}
  \hspace{1cm}
 \begin{tikzpicture}[thick,scale=0.5]
    \node [v] (bot1) {};
    \node [v] (bot2) [right of=bot1]{};
    \node [v] (bot3) [right of=bot2]{};
    \node[v] (mid1) [above of=bot1]{}
    edge[->] (bot2);
    \node[v] (mid2) [right of=mid1][label=right:$v$] {}
    edge[->] (bot1)
    edge[->] (bot2);
   \node [v](mid3) [right of=mid2] {}
    edge[->] (bot3);
    \node [v] (top1) [above right of=mid1] {}
    edge[->,bend angle=20,bend right,color=red] (bot2)
    edge[->] (mid1)
    edge[->] (mid2);
    \node [v] (top2) [right of=top1] {}
    edge[->] (mid1)
    edge[->] (mid2)
    edge[->] (mid3)
    edge[->,bend angle=20,bend right,color=red] (bot1);
    \node [below of=bot2,node distance=3mm] {$D*v$};
\end{tikzpicture}
\caption{A local complementation.}
\label{fig:lc}
\end{figure}
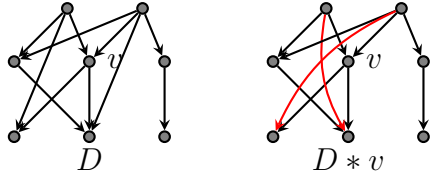

Note that $D*v*v=D$.
If $D$ is acyclic, then so is $D*v$.

\subsection*{Slide.}
For distinct vertices $u, v$ of a digraph $D$ with $N_D^-(u)=N_D^-(v)$ (possibly empty), we define $D\slides{u}{v}$ to be the digraph with $V(D\slides{u}{v})=V(D)$ and
$$
    A(D\slides{u}{v}) = A(D) \Delta \{(v,w) \mid w\in N_D^+(u)\}.
$$
Namely, when $N_D^-(u)=N_D^-(v)$, $(v,w)$ for $w\in N_D^+(u)$ is removed from $D$ if $w\in N_D^+(v)$ and added to $D$ otherwise.
We call it the \emph{slide} on $uv$.
See Figure~\ref{fig:local complement and slides}.
\begin{figure}
  \centering
  \tikzstyle{v}=[circle, draw, solid, fill=black!50, inner sep=0pt, minimum width=4pt]
  \tikzstyle{every edge}=[->,>=stealth,draw]
 \begin{tikzpicture}[thick,scale=0.5]
    \node [v] (bot1) {};
    \node [v] (bot2) [right of=bot1]{};
    \node [v] (bot3) [right of=bot2]{};
    \node[v] (mid1) [above of=bot1][label=left:$u$]   {}
    edge[->] (bot2);
    \node[v] (mid2) [right of=mid1][label=right:$v$] {}
    edge[->] (bot1)
    edge[->] (bot2);
   \node [v](mid3) [right of=mid2] {}
    edge[->] (bot3);
    \node [v] (top1) [above right of=mid1] {}
    edge[->] (bot1)
    edge[->] (mid1)
    edge[->] (mid2);
    \node [v] (top2) [right of=top1] {}
    edge[->] (mid1)
    edge[->] (mid2)
    edge[->] (mid3)
    edge[->] (bot2);
    \node [below of=bot2,node distance=3mm] {$D$};
\end{tikzpicture}
  \hspace{1cm}
 \begin{tikzpicture}[thick,scale=0.5]
    \node [v] (bot1) {};
    \node [v] (bot2) [right of=bot1]{};
    \node [v] (bot3) [right of=bot2]{};
    \node[v] (mid1) [above of=bot1][label=left:$u$]   {}
    edge[->] (bot2);
    \node[v] (mid2) [right of=mid1][label=right:$v$] {}
    edge[->] (bot1);
  \node [v](mid3) [right of=mid2] {}
    edge[->] (bot3);
    \node [v] (top1) [above right of=mid1] {}
    edge[->] (bot1)
    edge[->] (mid1)
    edge[->] (mid2);
    \node [v] (top2) [right of=top1] {}
    edge[->] (mid1)
    edge[->] (mid2)
    edge[->] (mid3)
    edge[->] (bot2);
   \node [below of=bot2,node distance=3mm] {$D\slides{u}{v}$};
 \end{tikzpicture}
  \hspace{1cm}
 \begin{tikzpicture}[thick,scale=0.5]
    \node [v] (bot1) {};
    \node [v] (bot2) [right of=bot1]{};
    \node [v] (bot3) [right of=bot2]{};
    \node[v] (mid1) [above of=bot1][label=left:$u$]   {}
    edge[->,color=red] (bot1);
    \node[v] (mid2) [right of=mid1][label=right:$v$] {}
    edge[->] (bot1)
    edge[->] (bot2);
   \node [v](mid3) [right of=mid2] {}
    edge[->] (bot3);
    \node [v] (top1) [above right of=mid1] {}
    edge[->] (bot1)
    edge[->] (mid1)
    edge[->] (mid2);
    \node [v] (top2) [right of=top1] {}
    edge[->] (mid1)
    edge[->] (mid2)
    edge[->] (mid3)
    edge[->] (bot2);
   \node [below of=bot2,node distance=3mm] {$D\slides{v}{u}$};
\end{tikzpicture}
\caption{A slide.}
 \label{fig:local complement and slides}
\end{figure}
We were not able to find a literature on slides.
Note that $D\slides{u}{v}\slides{u}{v}=D$ and $D\slides{u}{v}\neq D\slides{v}{u}$.
If $D$ is acyclic, then so is $D\slides{u}{v}$.

The acyclic digraph $D_A$ associated with a Bott matrix $A$ in $\T(n)$ has the canonical acyclic ordering $v_1,\dots,v_n$ of the vertices and one can easily check that
\[
D_{\Phik(A)}=D_A*v_k \quad\text{and}\quad D_{\Philk(A)}=D_A\slides{v_\ell}{v_m}.
\]
This means that the operation $\Phik$ in \eqref{2nd} corresponds to the local complementation at $v_k$ and the operation $\Philk$ in \eqref{3rd} corresponds to the slide on $v_{\ell}v_m$.

We say that two digraphs are \emph{Bott equivalent} if one is transformed into an isomorphic copy of the other through successive application of local complementations and slides.
The above observation shows that the correspondence $A\to D_A$ gives a bijective correspondence between Bott equivalence classes in $\T(n)$ and Bott equivalence classes of acyclic digraphs on $n$ vertices.

\begin{exam}
    There are $2$ non-isomorphic acyclic digraphs on $2$ vertices and they are not Bott equivalent.
    There are $6$ non-isomorphic acyclic digraphs on $3$ vertices and they are classified into four Bott equivalence classes listed in Figure~\ref{Fig:3nodes_class}. (Compare with Example~\ref{matrix34}.)

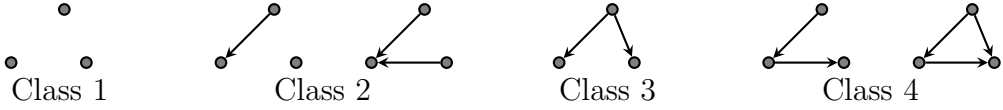
\begin{figure}
    \centering
\tikzstyle{v}=[circle, draw, solid, fill=black!50, inner sep=0pt, minimum width=4pt]
\tikzstyle{every edge}=[->,>=stealth,draw]
\begin{tikzpicture}[thick,scale=0.5]
    \node [v] (to) {};
    \node [v] (le) [below left of=to]{};
    \node [v] (ri) [right of=le]{};
   \node [below left of=ri,node distance=5mm] {Class $1$};
\end{tikzpicture}
  \hspace{1cm}
\begin{tikzpicture}[thick,scale=0.5]
    \node [v] (to) {};
    \node [v] (le) [below left of=to]{}
    edge[<-] (to);
    \node [v] (ri) [right of=le]{};
    \node [v] (le2) [right of=ri]{};
    \node [v] (ri2) [right of=le2]{}
    edge[->] (le2);
    \node [v] (to2) [above right of =le2]{}
    edge[->] (le2);
    \node [below right of=ri,node distance=5mm] {Class $2$};
\end{tikzpicture}
  \hspace{1cm}
\begin{tikzpicture}[thick,scale=0.5]
    \node [v] (to) {};
    \node [v] (le) [below left of=to]{}
    edge[<-] (to);
    \node [v] (ri) [right of=le]{}
    edge[<-] (to);
    \node [below left of=ri,node distance=5mm] {Class $3$};
\end{tikzpicture}
  \hspace{1cm}
\begin{tikzpicture}[thick,scale=0.5]
    \node [v] (to) {};
    \node [v] (le) [below left of=to]{}
    edge[<-] (to);
    \node [v] (ri) [right of=le]{}
    edge[<-] (le);
    \node [v] (le2) [right of=ri]{};
    \node [v] (ri2) [right of=le2]{}
    edge[<-] (le2);
    \node [v] (to2) [above right of =le2]{}
    edge[->] (le2)
    edge[->] (ri2);
    \node [below right of=ri,node distance=5mm] {Class $4$};
\end{tikzpicture}
  \caption{Four Bott equivalence classes of acyclic digraphs on $3$ vertices.}\label{Fig:3nodes_class}
\end{figure}
\end{exam}

    We list the number $\Diff{n}$ of Bott equivalence classes of acyclic digraphs on $n$ vertices up to $n=8$ in Table~\ref{tab:Bott equivalence} of Section~\ref{sect:intro}.
    Note that $\Diff{n}$ is in between $2^{(n-2)(n-3)/2}$ (see Example~\ref{Deltan}) and the number of non-isomorphic acyclic digraphs on $n$ vertices counted by Robinson~\cite{Rob-77}.

\section{Affine diffeomorphisms} \label{sect:affine}

In this section we associate an affine diffeomorphism between real Bott manifolds to each operation introduced in Section~\ref{sect:matrix}, and prove the implication $(1)\Rightarrow (2)$ in Theorem~\ref{main}.
We restate it for convenience as follows.

\begin{prop} \label{proposition:5.1}
    If Bott matrices $A$, $B$ in $\T(n)$ are Bott equivalent, then the associated real Bott manifolds $M(A)$ and $M(B)$ are affinely diffeomorphic.
\end{prop}
\begin{proof}
It suffices to find a group isomorphism $\phi\colon G(B)\to G(A)$ and a $\phi$-equivariant affine automorphism $\f$ of $T^n$ which induces an affine diffeomorphism from $M(B)$ to $M(A)$.
We may assume that $B=\Phi_P(A)$, $B=\Phik(A)$, or $B=\Philk(A)$.

\medskip
{\it The case of the operation $\OP$.}
Suppose $B=\Phi_P(A)$ for a permutation matrix~$P$ of a permutation $\sigma$ on $\{1,2,\ldots,n\}$.

We define a group isomorphism $\phi_P\colon G(B)\to G(A)$ by
\begin{equation} \label{phiS}
    \phi_P(b_{\sigma(i)}):=a_{i}
\end{equation}
and an affine automorphism $\f_P$ of $T^n$ by
\[
    \f_P(z_1,\dots,z_n):=(z_{\sigma(1)},\dots,z_{\sigma(n)}).
\]
Then it follows from \eqref{ai} (applied to $b_{\sigma(i)}$) that the $j$-th component of $\f_P(b_{\sigma(i)}(z))$ ($z \in T^n$) is $z_{\sigma(j)}(B^{\sigma(i)}_{\sigma(j)})$ for $j\neq i$ and $-z_{\sigma(i)}$ for $j=i$ while that of $a_i(\f_P(z))$ is $z_{\sigma(j)}(A^i_j)$ for $j\neq i$ and $-z_{\sigma(i)}$ for $j=i$.
Since $A^i_j=B^{\sigma(i)}_{\sigma(j)}$ by \eqref{SA=BA}, this shows that $\f_P$ is $\phi_P$-equivariant.

It follows from Lemma~\ref{f*} and \eqref{phiS} that the affine diffeomorphism $f_P\colon M(B)\to M(A)$ induced from $\f_P$ satisfies
\begin{equation} \label{op1coho}
    f_P^*(x_j)=y_{\sigma(j)}\quad \text{for all $j\in \{1,2,\ldots,n\}$.}
\end{equation}

\medskip
{\it The case of the operation $\OC$.}
Suppose $B=\Phik(A)$.
We define a group isomorphism $\phik\colon G(B)\to G(A)$ by
\begin{equation} \label{phik}
    \phik(b_i):=a_ia_k^{A^i_k}
\end{equation}
and an affine automorphism $\fk$ of $T^n$ by
\[
    \fk(z_1,\dots,z_n):=(z_1,\dots,z_{k-1},\sqrt{-1}z_k,z_{k+1},\dots,z_n).
\]

We shall check that $\fk$ is $\phik$-equivariant, that is
\begin{equation} \label{fkb}
    \fk(b_i(z))= a_ia_k^{A^i_k}(\fk(z)) \quad\text{for all $i\in \{1,2,\ldots,n\}$ and $z\in T^n$}.
\end{equation}
The identity is obvious when $i=k$ because $A^k_k=0$ and $B^k_j=A^k_j$ for every $j$ by \eqref{2nd}.
Suppose $i\neq k$.
Then the $j$-th component of the left hand side of \eqref{fkb} is given by
\[
\begin{cases}z_j(B^i_j)\quad&\text{for $j\neq i,k$},\\
-z_i \quad&\text{for $j=i$},\\
\sqrt{-1}(z_k(B^i_k)) \quad&\text{for $j=k$},
\end{cases}
\]
while that of the right hand side of \eqref{fkb} is given by
\[
\begin{cases}z_j(A^i_j+A^k_jA^i_k)\quad&\text{for $j\neq i,k$},\\
-z_i(A^k_iA^i_k) \quad&\text{for $j=i$},\\
(-1)^{A^i_k}(\sqrt{-1}z_k)(A^i_k) \quad&\text{for $j=k$}.
\end{cases}
\]
Since $B^i_j=A^i_j+A^k_jA^i_k$ by \eqref{2nd}, the $j$-th components above agree for $j\neq i,k$.
They also agree for $j=i$ because either $A^k_i$ or $A^i_k$ is zero.
We note that $B^i_k=A^i_k$ by \eqref{2nd} because $A^k_k=0$.
Therefore the $k$-th components above are both $\sqrt{-1}z_k$ when $B^i_k=A^i_k=0$ and $\sqrt{-1}\bar{z_k}$ when $B^i_k=A^i_k=1$.
Thus the $j$-th components above agree for every $j$.

It follows from Lemma~\ref{f*} and \eqref{phik} that the affine diffeomorphism $f_k\colon M(B)\to M(A)$ induced from $\fk$ satisfies
\begin{equation} \label{op2coho}
    (f_k)^*(x_j)=y_j\quad\text{for $j\neq k$}, \quad \quad (f_k)^*(x_k)=y_k+\sum_{i=1}^nA^i_ky_i.
\end{equation}

\medskip
{\it The case of the operation $\OR$.}
Suppose that $B=\Philk(A)$.
We define a group isomorphism $\philk\colon G(B)\to G(A)$ by
\begin{equation} \label{phiIC}
    \philk(b_i):=\begin{cases} a_\ell a_m\quad& \text{for $i = m$},\\
    a_i\quad&\text{for $i \neq m$,}
\end{cases}
\end{equation}
and an affine automorphism $\flk$ of $T^n$ by
$$
    \flk(z_1, \ldots, z_n):= (z_1, \ldots, z_{\ell-1}, z_\ell z_m, z_{\ell+1}, \ldots, z_n).
$$

We shall check that $\flk$ is $\philk$-equivariant.
To simplify notation we abbreviate $\flk$ and $\philk$ as $\f$ and $\phi$ respectively.
What we prove is the identity
\begin{equation} \label{fb=bf}
    \f(b_i(z))=\phi(b_i)\f(z) \quad \text{for all $i\in\{1,2,\ldots,n\}$ and $z\in T^n$}.
\end{equation}

Assume $i \neq \ell, m$.
Then the $j$-th component of the left hand side of \eqref{fb=bf} is given by
\[
\f(b_i(z))_j=\begin{cases}z_j(B^i_j)\quad&\text{for $j\neq i,\ell$},\\
-z_i \quad&\text{for $j=i$},\\
z_\ell(B^i_\ell) z_m(B^i_m) \quad&\text{for $j=\ell$},
\end{cases}
\]
while since $\phi(b_i)=a_i$ by \eqref{phiIC}, the $j$-th component of the right hand side of \eqref{fb=bf} is given by
\[
(\phi(b_i)\f(z))_j=\begin{cases}z_j(A^i_j)\quad&\text{for $j\neq i,\ell$},\\
-z_i \quad&\text{for $j=i$},\\
(z_\ell z_m)(A^i_\ell) \quad&\text{for $j=\ell$}.
\end{cases}
\]
This shows that  $\f(b_i(z))_j=(\phi(b_i)\f(z))_j$ for all $j$ because $B^i = A^i$ by \eqref{3rd} and $A_\ell= A_m$ by the condition on $A$ (hence $B^i_j=A^i_j$ for any $j$ and $B^i_\ell=A^i_\ell=A^i_m=B^i_m$), proving \eqref{fb=bf} when $i\neq \ell, m$.

Assume $i=\ell$. Then
\[
\f(b_\ell(z))_j=\begin{cases}z_j(B^\ell_j)\quad&\text{for $j\neq \ell$},\\
-z_m(B^\ell_m)z_\ell \quad&\text{for $j=\ell$},\\
\end{cases}
\]
while since $\phi(b_\ell) = a_\ell$ by \eqref{phiIC}, we have
\[
(\phi(b_\ell)\f(z))_j=\begin{cases}z_j(A^\ell_j)\quad&\text{for $j\neq \ell$},\\
-z_\ell z_m \quad&\text{for $j=\ell$}.\\
\end{cases}
\]
This shows that $\f(b_\ell(z))_j=(\phi(b_\ell)\f(z))_j$ for all $j$ because $B^\ell_j = A^\ell_j$ for any $j$ and $B^\ell_m=A^\ell_m=0$ by \eqref{3rd} and \eqref{eqn:masuda-added}, proving \eqref{fb=bf} when $i=\ell$.

Assume $i=m$.
Then
\[
\f(b_m(z))_j=\begin{cases}z_j(B^m_j)\quad&\text{for $j\neq \ell, m$},\\
-z_m \quad&\text{for $j=m$},\\
-z_\ell(B^m_\ell)z_m \quad&\text{for $j=\ell$},\\
\end{cases}
\]
while since $\phi(b_m) = a_\ell a_m$ by \eqref{phiIC}, we have
\[
(\phi(b_m)\f(z))_j=\begin{cases}z_j(A^\ell_j + A^m_j)\quad&\text{for $j\neq \ell, m$},\\
-z_m(A^\ell_m)\quad&\text{for $j=m$},\\
(-z_\ell z_m)(A^m_\ell) \quad&\text{for $j=\ell$}.\\
\end{cases}
\]
This shows that  $\f(b_m(z))_j=(\phi(b_m)\f(z))_j$ for all $j$ because $B^m_j = A^\ell_j + A^m_j$ for any $j$, $A^\ell_m=0$, $B^m_\ell = A^\ell_\ell + A^m_\ell = 0$ and $A^m_\ell = 0$ by \eqref{3rd} and \eqref{eqn:masuda-added}, proving \eqref{fb=bf} when $i=m$.

It follows from Lemma~\ref{f*} and \eqref{phiIC} that the affine diffeomorphism $f^{\ell,m}\colon M(B)\to M(A)$ induced from $\flk$ satisfies
\begin{equation} \label{op3coho}
(f^{\ell,m})^*(x_j)= \begin{cases}y_\ell + y_m \quad&\text{for $j = \ell$,}\\
y_j\quad&\text{for $j \neq \ell$.}
\end{cases}
\end{equation}
\end{proof}

\section{Cohomology isomorphisms} \label{sect:cohom}

In this section we prove the last part of Theorem~\ref{main} and the implication (3) $\Rightarrow$ (1) at the same time, summarized in the following proposition.

\begin{prop} \label{MAMBcoho}
Let $A$, $B$ be Bott matrices in $\T(n)$.
Every graded ring isomorphism \[H^*(M(A);\ZZ)\to H^*(M(B);\ZZ)\] is induced from a composition of affine diffeomorphisms corresponding to the three operations $\OP$, $\OC$ and $\OR$.
\end{prop}

By Proposition~\ref{proposition:5.1}, we may assume through affine diffeomorphisms corresponding to $\OP$ that our Bott matrices are strictly upper triangular.
We introduce a notion and prepare a lemma.
Remember that
\begin{equation} \label{HMA}
H^*(M(A);\ZZ)=\ZZ[x_1,\dots,x_n]/(x_j^2=x_j\sum_{i=1}^nA^i_jx_i\mid j=1,\dots,n).
\end{equation}
Note that $\{x_1, \ldots, x_n\}$ is a basis of $H^1(M(A);\ZZ)$ as a vector space. More generally, one easily sees that if $A$ is strictly upper triangular, then products $x_{i_1}\dots x_{i_q}$ $(1\le i_1<\dots<i_q\le n)$ form a basis of $H^q(M(A);\ZZ)$ as a vector space over $\ZZ$ so that the dimension of $H^q(M(A);\ZZ)$ is $\binom{n}{q}$ (see \cite[Lemma 5.3]{ma-pa08}).

We set
\[
\alpha_j=\sum_{i=1}^nA^i_jx_i\quad\text{for $j\in \{1,2,\ldots,n\}$}
\]
where $\alpha_1=0$ since $A$ is a strictly upper triangular matrix.
Then the relations in \eqref{HMA} are written as
\begin{equation} \label{alpha}
x_j^2=\alpha_jx_j \quad\text{for $j\in \{1,2,\ldots,n\}$.}
\end{equation}
Motivated by this identity we introduce the following notions which play an important role in the proof of Proposition~\ref{MAMBcoho}.

\begin{defi}
We say that an element $\alpha\in H^1(M(A);\ZZ)$ is an {\it eigen-element} of $H^*(M(A);\ZZ)$ if there exists $x\in H^1(M(A);\ZZ)$ such that
\[x^2=\alpha x,  \quad x\neq 0, \quad \text{ and } \quad x\neq\alpha.\]
The \emph{eigen-space} of $\alpha$, denoted by $\EA(\alpha)$, is the set of all elements $x\in H^1(M(A);\ZZ)$ satisfying the equation
\[x^2=\alpha x.\]
Clearly $\EA(\alpha)$ is a vector subspace of $H^1(M(A);\ZZ)$.
We also introduce a notation $\tEA(\alpha)$ which is the quotient of $\EA(\alpha)$ by the subspace spanned by $\alpha$, and call it the {\it reduced eigen-space} of $\alpha$.
\end{defi}

Eigen-elements and (reduced) eigen-spaces are invariants preserved under graded ring isomorphisms.
By \eqref{alpha}, $\alpha_1,\alpha_2,\ldots,\alpha_n$ are eigen-elements of $H^*(M(A);\ZZ)$ and the following lemma shows that these are the only eigen-elements.

\begin{lemm} \label{EAa}
Let $A$ be a strictly upper triangular Bott matrix in $\T(n)$.
Let $\alpha_j=\sum_{i=1}^n A^i_j x_i$ for all $j=1,2,\ldots,n$.
If $\alpha$ is an eigen-element of $H^1(M(A);\ZZ)$, then $\alpha=\alpha_j$ for some $j$ and the eigen-space $\EA(\alpha)$ of $\alpha$ is the span of $\{x_i\mid \alpha_i=\alpha, i\in\{1,\ldots,n\}\}\cup\{\alpha\}$.
\end{lemm}

\begin{proof}
By definition there exists a non-zero element $x\in H^1(M(A);\ZZ)$ different from $\alpha$ such that $x^2=\alpha x$.
Since both $x$ and $x+\alpha$ are non-zero, there exist $i$ and $j$ such that $x=x_i+p_i$ and $x+\alpha=x_j+q_j$ where $p_i$ is a linear combination of $x_1,\dots,x_{i-1}$ and $q_j$ is a linear combination of $x_1,\dots,x_{j-1}$.
Then
\begin{equation}
\label{eqn:masuda-add2}\alpha=x_i + x_j + p_i + q_j
\end{equation}
    and
\begin{equation}
x_ix_j+x_iq_j+x_jp_i+p_iq_j=0, \label{eqn:masuda-add2-2}
\end{equation}
where \eqref{eqn:masuda-add2-2}  follows from $x(x+\alpha)=0$.
As remarked above, products $x_{i_1}x_{i_2}$ $(1\le i_1<i_2\le n)$ form a basis of $H^2(M(A);\ZZ)$, and therefore $i=j$ for \eqref{eqn:masuda-add2-2} to hold.
Then, since $x_j^2=x_j\alpha_j$ and $\alpha_j$ is a linear combination of $x_1, \ldots, x_{j-1}$, it follows from \eqref{eqn:masuda-add2-2}  that $\alpha_j=q_j+p_i$ (and $p_iq_j=0$).
This together with \eqref{eqn:masuda-add2} shows that $\alpha=\alpha_j$, proving the first statement of the lemma.

Let $W$ be the span of $\{x_i\mid \alpha_i=\alpha, i\in\{1,\ldots,n\}\}\cup \{\alpha\}$ in $H^1(M(A);\ZZ)$.
Clearly $W\subseteq \EA(\alpha)$.
To prove the second statement of the lemma, it suffices to show that every element $x$ of $\EA(\alpha)$ is in $W$.
Since $\{x_1,\ldots,x_n\}$ is a basis,  $x=\sum_{i=1}^nc_ix_i$ and $\alpha=\sum_{i=1}^n d_i x_i$ for some $c_1,\ldots,c_n,d_1,\ldots,d_n\in \ZZ$.
We proceed by the induction on the number $N(x)$ of non-zero $c_i$'s.
If $N(x)=0$, then it is trivial. So we may assume that there exists $m$ such that $c_m=1$.

If $d_m=1$, then we express $x(x+\alpha)$ as a linear combination of the basis elements $x_{i_1}x_{i_2}$ $(1\le i_1<i_2\le n)$.
Since $c_m+d_m=0$ in $\ZZ$,  $x+\alpha$ is a linear combination of $x_i$'s with $i\neq m$.
Therefore the term in $x(x+\alpha)$ which contains $x_m$ is $x_m(x+\alpha)$ and it must vanish because $x(x+\alpha)=0$.
Therefore $x=\alpha$ and thus $x\in W$.

If $d_m=0$, the term in $x(x+\alpha)$ which contains $x_m$ is $x_m(x_m+\alpha)=x_m(\alpha_m+\alpha)$ since $x_m^2=\alpha_m x_m$, and it must vanish because $x(x+\alpha)=0$.
Therefore $\alpha_m=\alpha$ and hence $x_m \in \EA(\alpha)$ and $x_m\in W$.
Since $\EA(\alpha)$ is a vector space, $x+x_m\in \EA(\alpha)$.
Observe that $N(x+x_m)=N(x)-1$ and therefore by the induction hypothesis, $x+x_m\in W$.
This implies that $x\in W$.
\end{proof}

We are now ready to prove Proposition~\ref{MAMBcoho}.
\begin{proof}[Proof of Proposition~\ref{MAMBcoho}]
As remarked before, we may assume that both $A$ and $B$ are strictly upper triangular.
Let $x_1,x_2,\ldots,x_n$ be the canonical basis of $H^*(M(A);\ZZ)$ and let $y_1,y_2,\ldots,y_n$ be the canonical basis of $H^*(M(B);\ZZ)$.
Let $\alpha_j=\sum_{i=1}^n A^i_j x_i$, $\beta_j=\sum_{i=1}^n B^i_jy_i$ for $j\in \{1,2,\ldots,n\}$.

Let $\varphi\colon H^*(M(A);\ZZ)\to H^*(M(B);\ZZ)$ be a graded ring isomorphism.
It preserves the eigen-elements and (reduced) eigen-spaces.
In the following we shall show that we can change $\varphi$ into the identity map by composing isomorphisms induced from affine diffeomorphisms corresponding to the operations $\OP$, $\OC$ and $\OR$.

Through an affine diffeomorphism corresponding to the operation $\OP$ we may assume that $\varphi(\alpha_j)=\beta_j$ for all $j$ because of \eqref{op1coho}.
Then $\varphi$ restricts to an isomorphism $\EA(\alpha_j) \to \EB(\beta_j)$ between eigen-spaces and induces an isomorphism $\tEA(\alpha_j)\to \tEB(\beta_j)$ between reduced eigen-spaces.

Let $\alpha$ and $\beta$ be eigen-elements of $H^*(M(A);\ZZ)$ and $H^*(M(B);\ZZ)$, respectively.
Suppose that $\varphi(\alpha)=\beta$.
Let
\[J=\{j \mid \alpha_j=\alpha, j\in\{1,2,\ldots,n\} \}.\]
Let $\bar x_j$ be the image of $x_j$  in $\tEA(\alpha)$ and let $\bar y_j$ be the image of $y_j$ in $\tEB(\beta)$.
By Lemma~\ref{EAa}, $\{\bar x_j\mid  j\in J\}$ is a basis of $\tEA(\alpha)$ and $\{\bar y_j \mid j\in J\}$ is  a basis of $\tEB(\beta)$.
Thus if we express
$$
    \varphi(\bar x_j)=\sum_{i\in J}F^i_j\bar y_i  \quad \text{ for $j \in J$}
$$ with $F^i_j\in \ZZ$, then the matrix $F_J:=(F^i_j)_{i,j\in J}$ is invertible.

Since $\alpha_j=\sum_{i=1}^nA^i_jx_i$, we have $A_\ell=A_m$ if and only if $\alpha_\ell=\alpha_m$.
Therefore we can apply affine diffeomorphisms corresponding to the operation $\OR$ to every pair of distinct $\ell, m$ in $J$.
Let $A' = \Philk(A)$ and let $f= f^{\ell,m} \colon M(A') \to M(A)$ be the affine diffeomorphism considered in the previous section.
Let $x_1',\ldots, x_n'$ be the canonical generators of $H^\ast(M(A');\Z/2)$.
Then it follows from \eqref{op3coho} that if we express
$$
    (\varphi \circ (f^{-1})^\ast)(\bar{x_j'}) = \sum_{i \in J} F'^i_j \bar{y_i} \quad \text{ for $j \in J$,}
$$ then the matrix $F'_J = (F'^i_j)_{i,j\in J}$ is obtained from $F_J$ by adding $m$-th column to $\ell$-th column.
Similarly, an affine diffeomorphism corresponding to the operation $\OP$ induces a permutation of columns of $F_J$ by \eqref{op1coho}.
Since $F_J$ is an invertible binary matrix, one can change it to the identity matrix by permuting columns and adding a column to another column.
Therefore, through a sequence of affine diffeomorphisms corresponding to the operations $\OP$ and $\OR$, we may assume that $F_J$ is the identity matrix.
This can be done for each $J$ so that we may assume that
\[\varphi(x_j)=y_j \text { or }y_j+\beta_j \quad \text{ for every }j\in
\{1,2,\ldots,n\}.\]
Finally, through an affine diffeomorphism corresponding to the operation $\OC$, we may assume that $\varphi(x_j)=y_j$ for every $j$ by \eqref{op2coho}.
This means that after a successive application of the operations $\OP, \OC$ and $\OR$, we reach $A=B$ because $\varphi(\alpha_j)=\beta_j$, $\alpha_j=\sum_{i=1}^nA^i_jx_i$ and $\beta_j=\sum_{i=1}^nB^i_jy_i$ for every $j\in \{1,2,\ldots,n\}$, proving the proposition.
\end{proof}

\section{Unique Decomposition of Real Bott manifolds}\label{sect:decom}

We say that a real Bott manifold is \emph{indecomposable} if it is not diffeomorphic to a product of more than one real Bott manifolds.
The purpose of this section is to provide a graph theoretical proof of Theorem~\ref{main1} in Section~\ref{sect:intro}.
An algebraic proof can be found in \cite{masu08}.

The \emph{disjoint union} $D_1 \oplus D_2$ of two digraphs $D_1$ and $D_2$ is a digraph on the disjoint union of $V(D_1)$ and $V(D_2)$ such that $A(D_1 \oplus D_2) = A(D_1) \cup A(D_2)$.
Obviously, both $D_1$ and $D_2$ are acyclic if and only if so is $D_1  \oplus  D_2$.
A digraph is \emph{connected} if it is connected as an undirected graph.
An acyclic digraph $D$ is \emph{indecomposable} if all acyclic digraphs Bott equivalent to $D$ are connected.

For example, in Figure~\ref{Fig:dec}, three acyclic digraphs $D_1$, $D_2$, and $D_3$ are shown.
Obviously these digraphs are Bott equivalent.
Notice that $D_3$ is connected but $D_1$ and $D_2$ are disconnected.
Thus connectedness is not necessarily preserved under Bott equivalence.
Each component of $D_1$ and $D_2$ is indecomposable and yet $\{1,2\}\neq \{2,3\}$.
Hence a decomposition of an acyclic digraph into indecomposable acyclic digraphs does not induce a unique partition of the vertex set.

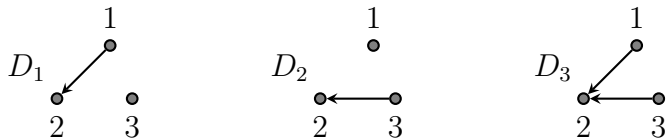
\begin{figure}
    \centering
\tikzstyle{v}=[circle, draw, solid, fill=black!50, inner sep=0pt, minimum width=4pt]
\tikzstyle{every edge}=[->,>=stealth,draw]
\begin{tikzpicture}[thick,scale=0.5]
    \node [v] (to) {};
   \node [v] (le) [below left of=to]{}
    edge[<-] (to);
  \node [v] (ri) [right of=le]{};
    \node [above] at (to.north) {$1$};
    \node [below] at (le.south) {$2$};
   \node [below] at (ri.south) {$3$};
   \node [above left] at (le.north) {$D_1$};
   \node [v] (le2) [right of=ri,node distance=2.5cm]{};
    \node [v] (ri2) [right of=le2]{}
    edge[->] (le2);
    \node [v] (to2) [above right of =le2]{};
    \node [above] at (to2.north) {$1$};
    \node [below] at (le2.south) {$2$};
   \node [below] at (ri2.south) {$3$};
   \node [above left] at (le2.north) {$D_2$};
   \node [v] (le3) [right of=ri2,node distance=2.5cm]{};
    \node [v] (ri3) [right of=le3]{}
    edge[->] (le3);
    \node [v] (to3) [above right of=le3]{}
    edge[->] (le3);
    \node [above] at (to3.north) {$1$};
    \node [below] at (le3.south) {$2$};
   \node [below] at (ri3.south) {$3$};
   \node [above left] at (le3.north) {$D_3$};
\end{tikzpicture}
\caption{Three Bott equivalent acyclic digraphs $D_1$, $D_2$, and $D_3$.}\label{Fig:dec}
\end{figure}

Surprisingly the next theorem will show that a decomposition of an acyclic digraph into indecomposable acyclic digraphs is unique up to Bott equivalence.
Note that a real Bott manifold $M(A)$ for a Bott matrix $A \in \T(n)$ is indecomposable if and only if the acyclic digraph associated to $A$ is indecomposable.
Therefore Theorem~\ref{main1} is equivalent to the following theorem.

\begin{theo} \label{thm:UDP}
Suppose that $D_1,D_2,\ldots,D_k$ and $H_1,H_2,\ldots,H_\ell$ are indecomposable acyclic digraphs.
If $\bigoplus_{i=1}^k D_i$ is Bott equivalent to $\bigoplus_{j=1}^\ell H_j$, then $k=\ell$ and there is a permutation $\sigma$ on $\{1, \ldots, k \}$ such that $D_i$ is Bott equivalent to $H_{\sigma(i)}$  for each $i=1, \ldots, k$.
\end{theo}

We prepare several lemmas before the proof of Theorem~\ref{thm:UDP}.
A vertex of a digraph $D$ is called a \emph{root} if it has no in-neighbor.
In other words, a vertex $v$ of $D$ is a root if and only if the column vector in $A_D$ corresponding to $v$ is zero.
Let $L_0(D)$ be the set of roots of $D$.

Obviously local complementations do not change connected components.
It is easy to check that slides at non-roots do not change connected components as well, because if we want to slide $uv$ for non-roots $u$ and $v$, then $u$ and $v$ must share a common in-neighbor.
The only trouble that might arise is that slides at roots may change connected components as we have seen in Figure~\ref{Fig:dec}.

The following lemma is easy to check.

\begin{lemm} \label{lemma:simplify}
Let $D$ be an acyclic digraph.
Let $u,v$ be distinct roots of $D$.
\begin{enumerate}[(i)]
\item
If $N^-_D(x)=N^-_D(y)$
for distinct $x,y \in V(D)\setminus L_0(D)$,
then
\[
 D\slides{x}{y}\slides{u}{v} = D \slides{u}{v} \slides{x}{y}.
\]
\item
For each vertex $x$ of $D$,
we have
\[ D \ast x \slides{u}{v}=D \slides{u}{v} \ast x. \]
\end{enumerate}
\end{lemm}

By Lemma~\ref{lemma:simplify}, one can push all slides on $L_0(D)$ to the left.
So it suffices to consider only slides on $L_0(D)$ when we are concerned with the indecomposability of connected components of $D$.
Clearly slides do not change the set of roots.

For subsets $X$ and $Y$ of $V(D)$, we denote by $\cormatrix{X,Y}{D}$ the submatrix of $A_D$ whose rows correspond to $X$ and columns to $Y$, and by $\row\cormatrix{X,Y}{D}$ the vector subspace of
$(\Z/2)^{|Y|}$ generated by row vectors in $\cormatrix{X,Y}{D}$.
For a set $X$  of vertices of $D$, we write $D\setminus X$ to denote the subgraph obtained by deleting vertices in $X$ and all the edges incident with a vertex in $X$.
For a vertex $v$ of $D$, we simply write $D\setminus v$ for $D\setminus \{v\}$.

\begin{lemm} \label{lemm:D=H}
Let $D$ and $H$ be acyclic digraphs such that $V(D)=V(H)$.
Then $H$ can be obtained from $D$ by applying slides only on $L_0(D)$ if and only if the following two conditions hold.
\begin{enumerate}
\item $D\setminus L_0(D)=H\setminus L_0(H)$.
\item $L_0(D)=L_0(H)$ and $\row\cormatrix{L_0(D),V(D)\setminus L_0(D)}{D} =\row\cormatrix{L_0(H), V(H)\setminus L_0(H)}{H}$.
\end{enumerate}
\end{lemm}

\begin{proof}
The forward implication is trivial, because slides on roots are row additions in the matrix $\cormatrix{L_0(D),V(D)\setminus L_0(D)}{D}$.

To prove the converse, let us assume that (1) and (2) hold.
Let $Y=V(D)\setminus L_0(D)$.
If necessary, we can interchange two rows corresponding to $x$ and $y$ of $\cormatrix{L_0(D),Y}{D}$ by replacing  $D$ with $D\slides{x}{y}\slides{y}{x} \slides{x}{y}$.
Hence, all elementary row operations can be obtained by slides on roots. By (2), $\cormatrix{L_0(D),Y}{D}$ and $\cormatrix{L_0(H),Y}{H}$ have the same reduced row echelon form. Therefore, $\cormatrix{L_0(H),Y}{H}$ can be obtained from $\cormatrix{L_0(D),Y}{D}$ by elementary row operations. This together with (1) implies the lemma.
\end{proof}

\begin{lemm} \label{full rank}
Let $D$ be an acyclic digraph with at least two vertices.
If $D$ is indecomposable, then
\[\rank\cormatrix{L_0(D), V(D)\setminus L_0(D)}{D}=\lvert L_0(D)\rvert.\]
\end{lemm}

\begin{proof}
Since $D$ is connected, $D$ must have a non-root.
If $\rank\cormatrix{L_0(D), V(D)\setminus L_0(D)}{D}<\lvert L_0(D)\rvert$, then there is a set of row vectors whose sum is zero.
So by applying slides, we obtain an acyclic digraph $H$ having a vertex with no out-neighbors and therefore $H$ is disconnected.
This contracts to the assumption that $D$ is indecomposable.
\end{proof}

\begin{lemm}\label{lemm:indecom}
  Let $D$ be an indecomposable acyclic digraph.
  Let $Y=V(D)\setminus L_0(D)$.
  Let $G$ be an acyclic digraph.
  If $H$ is obtained from $D\oplus G$ by applying slides on
  $L_0(D\oplus G)$, then there exists a set $X$ of roots of $H$ such
  that $|X|=|L_0(D)|$ and $H[X\cup Y]$ is connected, where $H[X\cup
  Y]=H\setminus (V(H)\setminus (X\cup Y))$.
\end{lemm}
\begin{proof}
 If $D$ has a single vertex, then this is trivial.
 So we may assume that $D$ has non-roots.
 Let $Z=V(G)\setminus L_0(G)$.
 Since $H\setminus Z$ can be obtained from $(D\oplus G)\setminus Z$ by applying slides on $L_0((D\oplus G)\setminus Z)=L_0(D\oplus G)$, we deduce from Lemma~\ref{lemm:D=H} that
 \[\row\cormatrix{L_0(D\oplus G),Y}{D\oplus   G} =\row \cormatrix{L_0(H),Y}{H}.\]
 By Lemma~\ref{full rank}, $\rank\cormatrix{L_0(D),Y}{D}=|L_0(D)|$.
 Therefore there exists a subset $X$ of $L_0(H)$  such that $|X|=|L_0(D)|$ and the rows in $\cormatrix{X,Y}{H}$ are linearly independent.

 By considering an isomorphic copy of $H$, we may assume that $X=L_0(D)$.
 This implies that
 \[\row \cormatrix{L_0(H),Y}{H}  =\row \cormatrix{X,Y}{H} =\row \cormatrix{X,Y}{H[X\cup Y]}.\]
 Since $L_0(G)$ has no arcs to $Y$ in $D\oplus G$,
 \[
 \row\cormatrix{L_0(D\oplus G),Y}{D\oplus
   G}=\row\cormatrix{X,Y}{D}.\]
 Therefore $\row\cormatrix{X,Y}{D} =\row \cormatrix{X,Y}{H[X\cup Y]}$.

 Then $D$ and $H[X\cup Y]$ satisfy (1) and (2) of Lemma~\ref{lemm:D=H} and therefore $H[X\cup Y]$ can be obtained from $D$ by applying slides on $L_0(D)$.
 Since $D$ is indecomposable, $H[X\cup Y]$ must be connected.
\end{proof}

Now, we complete the proof of Theorem~\ref{thm:UDP}.

\begin{proof}[Proof of Theorem~\ref{thm:UDP}]
We claim that it is enough to consider the case when $H$ is obtained from $D$ through slides on $L_0(D)$.
Suppose there is a sequence of local complementations and slides to apply to $D$ to obtain $H$.
By Lemma~\ref{lemma:simplify}, we may assume that slides on $L_0(D)$ are done first.
Let $H'$ be the acyclic digraph obtained by applying all the slides on $L_0(D)$.
Then $H$ can be obtained from $H'$ by applying slides on non-roots and local complementations.
Since these operations do not change the connected components, $H$ and $H'$ must have the identical set of connected components up to Bott equivalence.
By reversing these slides and local complementations, we can observe that each component of $H'$ is indecomposable and therefore $H'$ is  the disjoint union of $\ell$ indecomposable acyclic digraphs.
This proves the claim.

Then by Lemma~\ref{lemm:D=H},
\begin{equation} \label{rowDH}
\begin{split}
L_0(D)&=L_0(H),\\
D\setminus L_0(D)&=H\setminus L_0(H),\\
\row\cormatrix{L_0(D),V(D)\setminus L_0(D)}{D}&
=\row\cormatrix{L_0(H),V(H)\setminus L_0(H)}{H}.
\end{split}
\end{equation}

Let $A_i=\cormatrix{L_0(D_i),V(D_i)\setminus L_0(D_i)}{D_i}$.
Then
\[
\rank \cormatrix{L_0(D),V(D)\setminus L_0(D)}{D}
=\rank
\begin{pmatrix}
  A_1 &&&0\\
  & A_2\\
  & &\ddots\\
  0&&&A_k
\end{pmatrix}=\sum_{i=1}^k \rank A_i.\]
Similarly if $B_i=\cormatrix{L_0(H_i),V(H_i)\setminus L_0(H_i)}{H_i}$, then
\[\rank \cormatrix{L_0(H),V(H)\setminus L_0(H)}{H}=\sum_{i=1}^\ell \rank B_i.\]
By \eqref{rowDH}, $\sum_{i=1}^k \rank A_i=\sum_{i=1}^\ell \rank B_i$.
By Lemma~\ref{full rank}, $\rank A_i=|L_0(D_i)|$ if $D_i$ has at least two vertices and therefore $|L_0(D)|-\sum_{i=1}^k \rank A_i$ is the number of isolated vertices in $D$.
By \eqref{rowDH},  $|L_0(D)|-\sum_{i=1}^k \rank A_i=|L_0(H)|-\sum_{i=1}^\ell \rank B_i$ and therefore $D$ and $H$ should have the same number of isolated vertices.
Let $s$ be the number of isolated vertices and we may assume that $D_1, \ldots, D_{k-s}$ and $H_1, \ldots, H_{\ell-s}$ have non-roots.
Lemma~\ref{lemm:indecom} implies that there exists a function $\sigma:\{1,2,\ldots,k-s\}\to \{1,2,\ldots,\ell-s\}$ such that for each $i\in\{1,2,\ldots,k-s\}$, $V(D_i)\setminus L_0(D_i)$ stays in one connected component $H_j$ of $H$ for some $j=\sigma(i)$ and moreover $|V(H_j)|\ge |V(D_i)|$.
Similarly $V(H_j)\setminus L_0(H_j)$ stays in one connected component $D_m$ of $D$ for some $m$ with $|V(D_m)|\ge |V(H_j)|$.
Since $V(D_i)\setminus L_0(D_i)\subseteq V(D_m)$, we conclude that $i=m$, $|V(D_i)|=|V(H_j)|$, and $V(D_i)\setminus L_0(D_i)=V(H_j)\setminus L_0(H_j)$.
We may assume that $V(D_i)=V(H_j)$ by permuting roots.
Then it is easy to observe that $D_i$ and $H_j$ satisfy (1) and (2) of Lemma~\ref{lemm:D=H} from \eqref{rowDH} and therefore $D_i$ and $H_{\sigma(i)}$ are Bott equivalent. Clearly $\sigma$ is injective because $D_i$ and $H_{\sigma(i)}$ must share non-roots.
Since $V(D)=V(H)$, $\sigma$ should be bijective and $k=\ell$.
\end{proof}

\section{Numerical invariants of real Bott manifolds} \label{sec:invariant}
In this section, we produce numerical invariants of real Bott manifolds $M(A)$ using the Bott matrix $A\in \T(n)$ or its associated acyclic digraph $D_A$.

\subsection{Type.}
Recall that $L_0(D)$ is the set of roots of $D$, that is, the set of vertices having no in-neighbors.
For $k\ge 1$, we define $L_k(D)$ to be the set of roots in $D\setminus \bigcup_{i=0}^{k-1} L_i(D)$.
Alternatively we may define $L_k(D)$ to be the set of vertices $v$ of $D$ such that a longest directed path ending at $v$ has exactly $k$ arcs in $D$.
We call $L_k(D)$ for $k\ge 0$ the \emph{$k$-th level set} of $D$ and the sequence
\[
(|L_0(D)|,|L_1(D)|,|L_2(D)|,\ldots,|L_{n-1}(D)|)
\]
the \emph{type} of $D$.
\begin{prop}\label{prop:level}
If $D$ and $H$ are Bott equivalent acyclic digraph, then $\abs{L_k(D)}=\abs{L_k(H)}$ for all nonnegative integer $k$.
In particular, Bott equivalent acyclic digraphs have the identical type.
\end{prop}
\begin{proof}
It is enough to show that both local complementations and slides do not change $L_i(D)$.
Let $w$ be a vertex in $L_i(D)$.
Then there is a longest directed path $P_w$ in $D$  ending at $w$ with exactly $i$ arcs.

Let us first show that $w\in L_i(D*v)$  for any $v\in V(D)$.
It is enough to show that $P_w$ is a path in $D*v$ as well, because, if so, then  a longest path in $D*v$ ending at $w$ will be a path in $(D*v)*v=D$ as well.
Suppose that $P_w$ is not a path in $D*v$.
Then the local complementation at $v$ must remove at least one arc $(x,y)$ of $P_w$ and therefore $(x,v)$ and $(v,y)$ are arcs of $D$.
Since $D$ is acyclic, $v$ is not on $P_w$.
Then by replacing the arc $(x,y)$ by a path $xvy$ in $P_w$, we can find a path longer than $P_w$ in $D$, contradictory to the assumption that $P_w$ is a longest path ending at $w$.
This proves the claim that $L_i(D)=L_i(D*v)$ for all $i$.

Now let us prove that $w\in L_i(D\slides{u}{v})$ for $u\neq v\in V(D)$ with $N_D^-(u)=N_D^-(v)$.
Again, it is enough to show that$D\slides{u}{v}$ has a path of length $i$ ending at $w$; because if $D\slides{u}{v}$ has a longer path ending at $w$, then so does $D$ by the fact that $(D\slides{u}{v})\slides{u}{v}=D$.
We may assume that  the slide along $uv$ removes at least one arc $(x,y)$ of $P_w$.
Then $v=x$ and both $(v,y)$ and $(u,y)$ are arcs of $D$.
Since $N_D^-(u)=N_D^-(v)$, we can replace $v$ by $u$ in $P_w$ to obtain a path of the same length in $D\slides{u}{v}$.
This completes the proof.
\end{proof}

The level sets of $D_A$ for a Bott matrix $A\in \T(n)$ can be described in terms of $A$ as follows.
We identify the $i$-th vertex $v_i$ of $D_A$ with $i$ for $i\in\{1,2,\ldots,n\}$.
Then $L_0(D_A)$ can be identified with $L_0(A):=\{ j\mid A_j=0\}$.
We define $A(1)$ to be the matrix obtained from $A$ by removing all $j$-th columns and $j$-th rows for all $j\in L_0(A)$.
Then $L_1(A):=\{ j\mid A(1)_j=0\}$ can be identified with $L_1(D_A)$.
Inductively we define $A(k)$ for $k\ge 2$ to be the matrix obtained from $A(k-1)$ by removing all $j$-th columns and $j$-th rows for all $j$ with $A(k-1)_j=0$.
Then $L_k(A):=\{j\mid A(k)_j=0\}$ can be identified with $L_k(D_A)$.

The type of $D_A$ can also be described in terms of $H^*(M(A);\ZZ)$ as follows.
First note that $|L_0(D_A)|$ agrees with the dimension of the $\ZZ$-vector space
\[
N(\MH_0)=\{ x\in \MH_0^1 \mid x^2=0\}
\]
where $\MH_0=H^*(M(A);\ZZ)$ and $\MH_0^1$ denotes the degree one part of $\MH_0$, i.e., $\MH_0^1=H^1(M(A);\Z/2)$.
We define $\MH_1$ to be the quotient graded ring of $\MH_0$ by the ideal generated by $N(\MH_0)$ and inductively define $\MH_k$ for $k\ge 2$ to be the quotient graded ring of $\MH_{k-1}$ by the ideal generated by $N(\MH_{k-1})$.
Then $|L_k(D_A)|$ agrees with $\dim N(\MH_k)$.

\medskip
Let us try to motivate the notion of types geometrically.
The \emph{twist number} of a real Bott tower is the number of nontrivial topological fibrations $M_j \to M_{j-1}$ in the iterated $\R P^1$ bundle \eqref{tower}. Why is the twist number well-defined, even though a real Bott manifold may be represented by many real Bott tower structures?
To justify this, suppose that $A$ is an upper triangular binary matrix and $M(A)$ has the iterated $\R P^1$ bundle as in \eqref{tower}. Then, a trivial fibration in \eqref{tower} corresponds to a root of $D_A$ as well as a zero-column of $A$. Since $|L_0(D_A)|$ is invariant under Bott equivalence by Proposition~\ref{prop:level}, the twist number of $M(A)$ is indeed a topological invariant, and therefore it is well-defined. We remark that the twist number of a complex Bott manifold is discussed in \cite{Ch-Su-2011}.

This observation leads us to consider another bundle structure of $M(A)$.
One can see that $M(A)$ admits an iterated bundle structure
\[
    M(A)=X_m \longrightarrow X_{m-1} \longrightarrow \cdots \longrightarrow X_1 \longrightarrow X_0=\{\text{a point}\},
\]
where $X_{k+1} \to X_{k}$ is an $(\R P^1)^{|L_k(D_A)|}$-bundle and $m$ is the number of non-zero components in the type of $D_A$.
Since the type is invariant under Bott equivalence  by Proposition~\ref{prop:level}, if $M(A)$ admits an iterated bundle structure each of whose fiber is the product of $\R P^1$s, then the height of the iterated bundle must be greater than or equal to $m$ and the dimension of the $k$th fibration is $|L_k(D_A)|$.

\subsection{Rank.} \label{subsection:rank}

The operations $\OP$, $\OC$ and $\OR$ in Section~\ref{sect:matrix} preserve the rank of a Bott matrix $A$ in $\T(n)$, denoted $\rank A$.
Therefore the matrix rank is an invariant of Bott equivalent acyclic digraphs.
Here is a geometrical meaning of $\rank A$.

\begin{prop} \label{sumbetti}
For $A\in \T(n)$,
\[
\quad \dim_\Q H^1(M(A);\Q)\leq n-\rank A\quad\text{and}\quad\sum_{i=0}^n \dim_{\Q} H^i(M(A);\Q)=2^{n-\rank A}.
\]
\end{prop}
\begin{proof}
Remember that $M(A)$ is the quotient of $T^n$ by the action of a finite group $G(A)$; see Section~\ref{sect:rbott}.
Therefore it follows from \cite[Theorem 2.4 in p.120]{bred72} that
\begin{equation*} \label{MATn}
H^i(M(A);\Q)=H^i(T^n;\Q)^{G(A)} \quad\text{for every $i$,}
\end{equation*}
where the right hand side above denotes the invariants of the induced $G(A)$-action on $H^*(T^n;\Q)$.
Then it is shown in \cite[Lemma 2.1]{Ishida-10} that
\begin{equation} \label{HiMA}
    \dim_{\Q} H^i(M(A);\Q) = \left|\{J\subseteq\{1, \ldots, n\} \mid \lvert J\rvert=i, \sum_{j\in J} A_j = 0\}\right|.
\end{equation}
In particular, $\dim_\Q H^1(M(A);\Q)$ agrees with the number of zero column vectors in $A$ which is less than or equal to $n-\rank A$, proving the inequality in the proposition.

It also follows from \eqref{HiMA} that
\[\sum_{i=0}^n \dim_{\Q} H^i(M(A);\Q) = |X|,\]
where $X=\{J\subseteq\{1,\ldots, n\} \mid \sum_{j\in J} A_j = 0\}$.
Since an element of $X$ corresponds to a vector in the null space of $A$ whose dimension is $n-\rank A$, we have $|X|=2^{n-\rank A}$, proving the equality in the proposition.
\end{proof}

\begin{rema}
A \emph{real toric variety} is the submanifold of a toric variety fixed by the canonical involution induced from the conjugation.
It is known that a real Bott manifold is a real toric variety \cite{ma-pa08}.
While this paper is under review, formulas for the rational Betti numbers of real toric varieties have been established by Suciu-Trevisan~\cite{Tre-2012,Su-Tr} and Choi-Park \cite{Ch-Pa-2010}.
Their formulas hold not only for the $\Q$-coefficient but also for a general coefficient ring which has $2$ as a unit.
Using their formulas, one can compute the Betti number of real Bott manifold with an arbitrary coefficient ring which has $2$ as a unit.
However, no efficient formula for the \emph{integral} (co)homology of real Bott manifolds is known.
In particular, we do not know whether, for $k \ge 2$, a real toric Bott manifold has a $2^k$-torsion which is not a $2^{k-1}$-torsion in its integral homology group.
\end{rema}

In 1985, Halperin~\cite{halperin85} conjectured that if a compact torus $T^k$ of dimension $k$ acts on a finite dimensional topological space $X$ \emph{almost freely}, i.e., any isotropy subgroup is finite, then
\[
\sum_{i=0}^{\dim X} \dim_\Q H^i(X;\Q)\ge 2^k.
\]
This conjecture is called the \emph{toral rank conjecture} or the \emph{Halperin-Carlsson conjecture}.
No counterexamples and some partial affirmative answers are known, see \cite{al-pu93} for example.
We show that the toral rank conjecture holds for real Bott manifolds.

\begin{theo} \label{HaCa}
Let $A\in\T(n)$.
If $M(A)$ admits an effective topological action of a torus $T^k$ of dimension $k$, then
$$\sum_{i=0}^n \dim_{\Q} H^i(M(A);\Q)\ge 2^k.$$
\end{theo}

\begin{proof}
Choose any point $p\in M(A)$ and consider a map $f_p\colon T^k\to M(A)$ defined by $f_p(t):=tp$.
Let $\pi_1(X)$ be the fundamental group of a topological space $X$ and let $Z(\pi_1(X))$ be the center of $\pi_1(X)$.
Since $M(A)$ is an aspherical manifold, $f_p$ induces an injective homomorphism $\pi_1(T^k)\to Z(\pi_1(M(A)))$ (which implies that the action of $T^k$ on $M(A)$ is almost free), see \cite{co-ra71}.
Kamishima and Nazra~\cite[Proposition 2.4]{ka-na08} have shown that the intersection of $Z(\pi_1(M(A)))$ and the commutator subgroup $[\pi_1(M(A)), \pi_1(M(A))]$ of $\pi_1(M(A))$ is a trivial subgroup of $\pi_1(M(A))$.
Hence, the natural map
$$
    \pi_1(T^k) \to Z(\pi_1(M(A))) \to \pi_1(M(A))/[\pi_1(M(A)), \pi_1(M(A))]=H_1(M(A);\Z)
$$ is injective.
It follows that
$$
k\le \dim_\Q H_1(M(A);\Q)=\dim_\Q H^1(M(A);\Q).
$$ Hence the theorem follows from Proposition~\ref{sumbetti}.
\end{proof}

\subsection{Odd out-degree vertices.}
The \emph{odd height} of an acyclic digraph $D$ is the length of a longest directed path ending at a vertex of odd out-degree.
In other words, it is the maximum $k$ such that $L_k(D)$ contains a vertex of odd out-degree.
If $D$ has no vertex of odd out-degree, then we define the odd height of $D$ to be $\infty$.

\begin{prop}
  Bott equivalent acyclic digraphs have the same odd height.
\end{prop}
\begin{proof}
Let $D$ be an acyclic digraph and let $v\in V(D)$.
Let $x\neq  y\in V(D)$ with $N_D^-(x)=N_D^-(y)$.
Then
 \begin{align}
  \label{D*v}
    \deg^+_{D*v}(w)&  \equiv
    \begin{cases}
     \deg^+_{D}(w)+\deg^+_D(v)&\pmod2   \quad\text{ if $w \in N^-_D(v)$,}\\
     \deg^+_D(w)& \pmod 2 \quad\text{ otherwise,}
    \end{cases}
    \\
    \label{Dslidexy}
    \deg^+_{D\slides{x}{y}}(w)&\equiv
    \begin{cases}
      \deg^+_D(w)+\deg^+_D(x)&\pmod 2 \quad\text{ if }w=y,\\
      \deg^+_D(w)&\pmod 2 \quad\text{ otherwise.}
    \end{cases}
  \end{align}
Therefore, if every vertex of $D$ has even out-degree, then a local complementation and a slide do not create a vertex of odd out-degree.
So we may assume that $D$ has a vertex of odd out-degree.
Let $k$ be the odd height of $D$ and let $w$ be a vertex of odd out-degree in $L_k(D)$.

Let us first consider $D*v$ for a vertex $v\in L_i(D)$.
If $w\in N_D^-(v)$, then $\deg^+_D(v)$ is even as $i>k$; so $\deg^+_{D*v}(w)$ is odd by \eqref{D*v}.
If $w\notin N_D^-(v)$, then $\deg^+_{D*v}(w)$ is again odd by \eqref{D*v}.
This proves that the odd height of $D*v$ is at least the odd height of $D$.
Since $(D*v)*v=D$, we conclude that $D$ and $D*v$ have the same odd height.

Now we claim that $L_k(D\slides{x}{y})$ has a vertex of odd out-degree in $D\slides{x}{y}$.
Suppose that $\deg^+_{D\slides{x}{y}}(w)$ is even.
Then $w=y$ and $\deg^+_D(x)$ is odd by \eqref{Dslidexy}.
We note that $x\in L_k(D)$ because $x$ and $y$ are in the same level set of $D$ and $w=y$ is in $L_k(D)$.
Therefore $x\in L_k(D)$ and $\deg^+_D(x)$ is odd, proving the claim.
Since $(D\slides{x}{y})\slides{x}{y}=D$, we conclude that $D$ and $D\slides{x}{y}$ have the same odd height.
\end{proof}

By Lemma~\ref{ori-sym} (i), for a Bott matrix $A$ in $\T(n)$, $M(A)$  is orientable if and only if the out-degree of every vertex of $D_A$ is even.
In other words, $M(A)$ is orientable if and only if the odd height of $D_A$ is $\infty$.
Hence, the notion of odd height may be thought of as a refinement of the orientability of real Bott manifolds.

\subsection{Sibling classes.}
For $x, y\in V(D)$, we say that $x\sim_D y$ if $N_D^-(x)=N_D^-(y)$.
Then $\sim_D$ is an equivalence relation on $V(D)$ and we call an equivalence class a \emph{sibling class} of $D$.
If $x\sim_D y$, then $x$ and $y$ are in the same level set; so each level $L_i(D)$ is partitioned into sibling classes.
We note that a sibling class of $D_A$ for a Bott matrix $A\in \T(n)$ corresponds to a maximal set of identical columns of $A$.
\begin{prop}
    Sibling classes are invariant under Bott equivalence.
\end{prop}
\begin{proof}
Let $x$ and $y$ be vertices in the same sibling class of an acyclic digraph $D$.
Since $N_D^-(x)=N_D^-(y)$, for every $w\in V(D)$, both $(w,x)$ and $(w,y)$ are arcs of $D$ or neither $(w,x)$ nor $(w,y)$ are arcs of $D$.

Firstly let us consider the case $D*v$ for $v\in V(D)$.
If both $(v,x)$ and $(v,y)$ are arcs of $D$, then one easily sees that $N_{D*v}^-(x)=N_{D*v}^-(y)$.
If both $(v,x)$ and $(v,y)$ are not arcs of $D$, then the set of in-neighbors of $x$ and $y$ remains unchanged under the local complementation at $v$.
Therefore  $N_{D*v}^-(x)=N_{D*v}^-(y)$ in any case.

Secondly let us consider the case $D\slides{u}{v}$ for $u\neq v\in V(D)$ with $N_D^-(u)=N_D^-(v)$.
If both $(u,x)$ and $(u,y)$ are not arcs of $D$, then the set of in-neighbors of $x$ and $y$ remains unchanged under the slide on $uv$.
Suppose that both $(u,x)$ and $(u,y)$ are arcs of $D$.
Then, $(v,x)$ (resp. $(v,y)$) is an arc of $D\slides{u}{v}$ if and only if $(v,x)$ (resp. $(v,y)$) is not an arc of $D$.
Therefore $N_{D\slides{u}{v}}(x)=N_{D\slides{u}{v}}(y)$ in any case.
\end{proof}

By Lemma~\ref{ori-sym} (ii), for a Bott matrix $A\in \T(n)$, $M(A)$ admits a symplectic form if and only if the cardinality of every sibling class of $D_A$ is even.
Hence, the notion of sibling class can be seen as a refinement of the symplecticity of real Bott manifolds.
It is easy to see that if the cardinality of each sibling class of $D$ is even, then the odd height of $D$ must be $\infty$.
This is obvious from the topological viewpoint because every symplectic manifold is orientable.

\subsection{Cut-rank.}
Recall that, for subsets $X$ and $Y$ of the vertex set of a digraph $D$, we write $\cormatrix{X,Y}{D}$ to denote the submatrix of the adjacency matrix of $D$ whose rows correspond to $X$ and columns correspond to $Y$.
Let $\rho_D(X)=\rank \cormatrix{X,V(D)\setminus X}{D}$.
This function $\rho_D\colon 2^{V(D)} \to \mathbb{Z}$ is called the \emph{cut-rank} function of $D$.

The cut-rank function naturally appeared when studying properties of local complementations.
In 1985, Bouchet~\cite{Bouchet1985} studied the cut-rank function on undirected graphs together with local complementation.
There are motivations based on matroid theory.
In fact, the cut-rank function of a undirected bipartite graph is the Tutte connectivity function of a binary matroid having that graph as a fundamental graph, see Oum~\cite{oum05}.
Moreover local complementations of undirected graphs can be discovered when trying to find appropriate graph operations to describe matroid minors.
This allowed generalizations of theorems on binary matroids to undirected graphs.
For digraphs, in 1987, Bouchet~\cite{Bouchet1987c} studied the cut-rank function of directed graphs and proved that local complementation on directed graphs preserves the cut-rank function.
The name ``cut-rank'' was first introduced in Oum and Seymour~\cite{ou-se06}.

For our application, the cut-rank function is not preserved under sliding.
But slides can  be applied only to the siblings, and therefore we can
prove that the cut-rank function on a union of level sets is preserved as follows.

\begin{prop}\label{prop:cutrank}
  Let $D$, $H$ be Bott equivalent acyclic digraphs on $n$ vertices.
  Then,
 \begin{enumerate}
  \item[(i)]   $\rho_D(\cup_{j\in J} L_j(D))=\rho_H(\cup_{j\in J} L_i(H))$ for every subset $J$ of $\{0,1,2,\ldots,n-1\}$,
  \item[(ii)]  $\rank\cormatrix{L_j(D),L_{j+1}(D)}{D}=\rank\cormatrix{L_j(H),L_{j+1}(H)}{H}$ for all $j\in \{0,1,2,\ldots,n-2\}$.
  \end{enumerate}
\end{prop}
\begin{proof}
  It is enough to prove when $H=D*v$ or $H=D\slides{u}{w}$.
  By Proposition \ref{prop:level}, $L_j(D)=L_j(H)$ for each $j$.

  (i)
  For a subset $J$ of $\{0,1,2,\ldots,n-1\}$, let $X=\cup_{j\in J} L_j(D)$ and $Y=V(D)\setminus X$.
  Let $M=\cormatrix{X,Y}{D}$ and $M'=\cormatrix{X,Y}{H}$.
  Then $\rho_D(X)=\rank M$ and $\rho_H(X)=\rank M'$.

  Let us first consider the case when $H=D*v$ for a vertex $v$ of $D$.
  Then either $v\in X$ or $v\in Y$.
  If $v\in X$, then $M$ has a row indexed by $v$ and $M'$ is obtained from $M$ by adding the row of $v$ to all rows indexed by in-neighbors of $v$ in $X$.
  If $v\in Y$, then $M$ has a column indexed by $v$ and $M'$ is obtained from $M$ by adding the column of $v$ to all columns indexed by out-neighbors of $v$ in $Y$.
  Thus, in both cases, $\rank M=\rank M'$.

  Now let us consider the case when $H=D\slides{u}{w}$ for two vertices $u$, $w$ having the same set of in-neighbors.
  Since $u$ and $w$ have the same set of in-neighbors, they belong to the same level.
  Therefore either $\{u,w\}\subseteq X$ or $\{u,w\}\subseteq Y$.
  If $\{u,w\}\subseteq Y$, then $M'=M$.
  If $\{u,w\}\subseteq X$, then $M'$ is obtained from $M$ by adding the row of $u$ to the row of $w$.
  So $\rank M=\rank M'$.
  This completes the proof of (i).

 (ii)
 Since $D$ is acyclic, there is no arc from $L_{a}(D)$ to $L_b(D)$ if $a>b$.
 Therefore
 \[
 \rank\cormatrix{L_j(D),L_{j+1}(D)}{D}= \rho_D(L_j(D)\cup L_{j+2}(D)\cup L_{j+3}(D)\cup \cdots \cup L_{n-1}(D)),\]
 and by (i), we have (ii).
\end{proof}

We remark that the invariants discussed in Section~\ref{sec:invariant} completely classify all Bott equivalence classes up to $4$ vertices but not on  $5$ vertices.
One can easily check that two acyclic digraphs in Figure~\ref{fig:unfortune case} are not Bott equivalent but have the same set of invariants.

\begin{figure}
  \centering
  \tikzstyle{v}=[circle, draw, solid, fill=black!50, inner sep=0pt, minimum width=4pt]
  \tikzstyle{every edge}=[->,>=stealth,draw]
  \begin{tikzpicture}
    \node [v] (r3) {};
    \node [v] [above of=r3] (r2) {}
    edge[->] (r3);
    \node [v] [above of=r2] (r1) {}
    edge[->] (r2);
    \node [v] [left of=r2] (l2) {};
    \node [v] [above of=l2] (l1) {}
    edge [->] (l2)
    edge [->] (r3);
  \end{tikzpicture}
  \hspace{1cm}
  \begin{tikzpicture}
    \node [v] (r3) {};
    \node [v] [above of=r3] (r2) {}
    edge[->] (r3);
   \node [v] [left of=r2] (l2) {};
    \node [v] [above of=l2] (l1) {}
    edge [->] (l2);
    \node [v] [above of=r2] (r1) {}
    edge[->] (r2)
    edge[->] (l2);
  \end{tikzpicture}
  \caption{Theses acyclic digraphs are not Bott equivalent but have the identical set of invariants discussed in this paper.}
  \label{fig:unfortune case}
\end{figure}
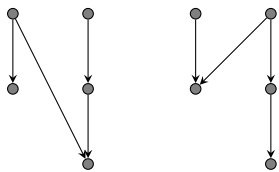

\section*{Acknowledgements}
The authors are grateful to Hanchul Park for his useful comments on Proposition~\ref{sumbetti}.
The authors are also thankful to Yoshinobu Kamishima for his comments on Theorem~\ref{HaCa}.

\end{document}